\documentclass[reqno,a4paper,11pt]{amsart}
\usepackage[utf8]{inputenc}
\usepackage{amsmath}
\usepackage{amssymb}
\usepackage{mathtools}
\usepackage{amsthm}

\usepackage{hyperref}

\usepackage{accents}

\newtheorem{theorem}{Theorem}[section]
\newtheorem{corollary}[theorem]{Corollary}
\newtheorem{lemma}[theorem]{Lemma}
\newtheorem{proposition}[theorem]{Proposition}

\theoremstyle{definition}
\newtheorem{definition}[theorem]{Definition}

\title{Tomography of $1$-forms on a gas giant}
\author{Joonas Ilmavirta}
\thanks{Department of Mathematics and Statistics, University of
  Jyv\"askyl\"a, Finland. \texttt{joonas.ilmavirta@jyu.fi}}
\author{Antti Kykkänen}
\thanks{Department of Computational Applied Mathematics and Operations Research, Rice University, Houston,
TX, USA \texttt{ak272@rice.edu}}
\author{Eetu Satukangas}
\thanks{Department of Mathematics and Statistics, University of
  Jyv\"askyl\"a, Finland. \texttt{eetu.a.e.satukangas@jyu.fi}}
\date{\today}

\newcommand{\R}{{\mathbb R}}
\newcommand{\N}{{\mathbb N}}
\newcommand{\abs}[1]{\left\lvert#1\right\rvert}

\newcommand{\der}{\mathrm{d}}
\newcommand{\Der}{\mathrm{D}}

\newcommand{\Ver}{\mathcal{V}}
\newcommand{\Hor}{\mathcal{H}}

\newcommand{\Zeta}{\mathcal{Z}}

\newcommand{\eps}{\varepsilon}

\newcommand{\ip}[2]{\left\langle#1,#2\right\rangle}
\newcommand{\teksti}[1]{\quad \text{{#1}}\quad }

\newcommand{\Order}{\mathcal{O}}

\newcommand{\grad}[1]{\overset{\mathrm{#1}}{\nabla}}
\newcommand{\gradv}{\grad{v}}
\newcommand{\gradh}{\grad{h}}

\newcommand{\Vol}{\mathrm{Vol}}

\newcommand{\norm}[1]{\Vert #1 \Vert}

\usepackage{caption} 
\begin{document}

\begin{abstract}
We show that on gas giant manifolds the geodesic X-ray transform is solenoidally injective on one-forms that are smooth up to the boundary in an appropriate smooth structure.
A gas giant manifold is a conformally blown up Riemannian manifold whose boundary singularity is milder than asymptotically hyperbolic.
The proof is based on a Pestov identity and asymptotic analysis of short geodesics.
\end{abstract}

\maketitle

\section{Introduction}
\label{sec:intro}

Let $(M,\bar g)$ a compact Riemannian manifold with boundary, and $\rho\colon M\to[0,\infty)$ a boundary defining function.
The corresponding \emph{gas giant manifold} is the same manifold $M$ equipped with the metric $g=\rho^{-1}\bar g$.

The propagation of acoustic waves on a gas giant planet is well modeled by the geodesic flow of a gas giant manifold as explained in~\cite{dHIKM2024}.
A natural question in inverse problems, the mathematics of indirect measurements, is to find the flow field of the gas from the associated Doppler effect on travel times of acoustic waves.
Linearizing this problem at zero flow leads to the ray transform of vector fields~\cite{Schuster:Doppler}.
We set out to study this problem in gas giant geometry.

We denote the geodesic X-ray transform on one-forms by $I$.
Vector fields can be identified with one-forms in the interior, but at the boundary the identification provided by the metric tensor blows up.
There are two natural choices of smooth structures on the manifold, one corresponding to boundary normal coordinates of~$\bar g$ and the other to those of~$g$.
We use the latter one, and smoothness of functions and tensor fields is understood with respect to this structure.
See section~\ref{subsec:choice-of-structure} for details.

\begin{theorem}
\label{thm:injectivity-mod-gauge}

Let~$(M,g)$ be a non-trapping\footnote{Non-trapping means that all geodesics starting from the interior reach the boundary in finite time.} gas giant with everywhere non-positive curvature. For a smooth (up to the boundary)~$1$-form~$f$ on~$M$ the following are equivalent:
\begin{enumerate}
    \item The geodesic X-ray transform of~$f$ vanishes.
    \item There is a smooth function~$p \in C^\infty(M)$ with~$p|_{\partial M}=0$ so that~$f = \der p$.
\end{enumerate}
\end{theorem}

The proof is based on a Pestov identity.
In order to be able to use this energy identity on a singular manifold, we need to ensure that the appropriate functions are integrable.
This is achieved by a high-order boundary determination based on asymptotic analysis of short geodesics.

\subsection{Choice of structures}
\label{subsec:choice-of-structure}

Our main results concern inverse problems in gas giant geometry.
A secondary point is the importance of the choice of structure, including a smooth structure, on our singular manifold.

There are two natural choices for the smooth structure of a gas giant manifold. One is rather obvious and gives an adequate description for the geodesic flow, and the other one, while less obvious, gives the cleanest description of the geodesic flow of the manifold on the cotangent bundle of the manifold.

The two choices of coordinates are best illustrated by considering a simple example. Let $\bar g=\der s^2 + \der y^2$ be the standard Euclidean metric in the half-plane~$\{(s,y):s>0\} \subset \R^2$ and consider the gas giant metric
\begin{equation}
\label{eqn:smooth-structure-1}
g = \frac{\der s^2 + \der y^2}{s}.
\end{equation}
The smooth coordinates~$(s,y)$ corresponds to a choice of smooth structure for the manifold (the obvious one), where~$s$ measures the~$\bar g$-distance to the boundary. However, the coordinate change~$s = x^2$ brings the metric essentially into the form
\begin{equation}
\label{eqn:smooth-structure-2}
g = \der x^2 + \frac{\der y^2}{x^2}.
\end{equation}
The mapping~$s \mapsto x^2$ is smooth, but its inverse~$x \mapsto s^{1/2}$ is singular at~$s = 0$. Hence the change of coordinates corresponds to introduction of a second smooth structure, where~$x$ measures the~$g$-distance to the boundary. This choice of smooth structure is less obvious at the first glance, but turns out to be cleaner for the analysis of the geodesic flow on the cotangent bundle.

Let the dual coordinates --- the coordinates induced on the cotangent bundle --- corresponding to~$(s,y)$ be denoted by~$(\xi_s,\eta)$ and similarly the dual coordinates for~$(x,y)$ be denoted by~$(\xi_x,\eta)$. The geodesic flow on the cotangent bundle is then given as the bicharacteristics of the Hamiltonian that is written as $H(s,y,\xi_s,\eta)=\frac{1}{2}(s\xi_s^2 +s\eta^2)$ or $H(x,y,\xi_x,\eta)=\frac{1}{2}(\xi_x^2+x^2\eta^2)$ in the corresponding choice of smooth structure. In table~\ref{tab:dot-xi} we have collected the second order evolutions of the normal component of geodesics for the two choices of the smooth structure, both in the  tangent and cotangent bundle versions of the geodesic flow. The main take away is that only in the~$g$-boundary normal structure and in the Hamiltonian formalism of the geodesic flow, the evolution is non-singular up to the boundary~$x = 0$. This indicates clearly that the best choice for the smooth structure is the one with smooth coordinates~$(x,y)$ and the geodesic flow on the cotangent bundle.

\begin{table}[h]
    \centering
    \begin{tabular}{|c|c|c|}
        \hline  & tangent bundle $TM$ & cotangent bundle $T^* M$ \\ \hline\hline
        $\bar g$-normal coordinates & $\ddot s = \frac{1}{2s}\big((\dot s)^2 + (\dot y)^2\big)$ & $\dot \xi_s = -\frac{1}{2}(\xi_s^2 + \eta^2) = -\frac{1}{s}H(s,y,\xi_s,\eta)$ \\ \hline
        $g$-normal coordinates & $\ddot x = \frac{1}{x^3}(\dot y)^2$ & $\dot \xi_x = -x\eta^2$ \\ \hline
    \end{tabular}
    \caption{The rows of the table correspond to the two possible choices of smooth structures for the gas giant. In the structure of the first row the coordinates~$(s,y)$ are smooth and in the structure of the second row the coordinates~$(x,y)$ are smooth (see~\eqref{eqn:smooth-structure-1} and~\eqref{eqn:smooth-structure-2}). On the columns, we have computed normal to the boundary components of the geodesic flow on tangent and cotangent bundle. The take away is that the only choice to have non-singular flow is to choose the smooth structure~$(x,y)$ and to consider the geodesic flow on the cotangent bundle.}
    \label{tab:dot-xi}
\end{table}

In~\cite{dHIKM2024}, the authors use the polytropic equation of state coupled with the Lane-Emden equation to derive a law that the metric on a gas giant is asymptotically proportional to~$\rho^{-1}$, where~$\rho$ represents the physical distance to the surface. Using this computation as a guideline the authors then defined a more general Riemannian model for gas giant geometry by
\begin{equation}
g = \frac{\bar g}{\rho^\alpha}
\end{equation}
where~$\alpha$ is a parameter in~$(0,2)$. This provides a more general and therefore a more flexible framework to develop the geometry in, but nevertheless the physical (predicted by the polytropic model) blow-up rate of the geometry is~$\rho^{-1}$. In the current article, we chose to take~$\alpha = 1$, which makes certain smoothness considerations simpler, while still remaining true for the physicality of the model. For example, in table~\ref{tab:dot-xi}, fractional powers of $x$ and $s$ would appear when~$\alpha$ is irrational, which leads to certain non-smoothness of the geodesic flow at the boundary in both possible choices for the smooth structure and for both tangent and cotangent bundle versions of the flow.

The proofs of the lemmas we present in section~\ref{sec:proof-of-the-main-result} are not valid as such for general $\alpha\in(0,2)$ due to singularities in the geodesic vector field and other quantities.

\subsection{Related results}

Our article continues the long tradition of studying inverse problems in integral geometry on Riemannian manifolds with a boundary.
The very origin of integral geometry is in the work of Funk~\cite{Funk1913} on the round sphere, but manifolds with boundary followed much later.
The original problem for the X-ray transform of functions and the method of Pestov identities can be traced back to the seminal work of Mukhometov~\cite{Mukhometov1975,Mukhometov1977,Mukhometov1981} and the problem for~$1$-forms at least back to~\cite{AR1997}.
The method of Pestov identities which is in use in this article too has been well-documented in several books and surveys (see e.g.~\cite{IM2019,PSU2014,PSU2023,Sharafutdinov1994}). As the state of art on modern tensor tomography on manifold with boundaries we mention~\cite{PSU2013} where it was shown that the X-ray transform of~$2$-tensor fields is solenoidally injective on simple Riemannian surfaces. The problem remains open on general simple manifolds in dimensions~$3$ and higher.

One of the more recent developments in ray tomography on manifolds has been to replace the standard Riemannian metric of the manifold with an asymptotically hyperbolic one.
In simplified terms, this is related to the setting of gas giants in that an asymptotically hyperbolic metric is of the form~$g = \rho^{-2}\bar g$ whereas a gas giant metric is~$g = \rho^{-1}\bar g$. The X-ray transform is relatively well-understood on asymptotically hyperbolic manifolds (see~\cite{Eptaminitakis2022,EG2021,EMZ2025a,EMZ2025b,GGSU2019}) alongside with several other inverse problems in asymptotically hyperbolic geometry (see e.g.~\cite{GSaB2009,Lefeuvre2020,MarxKuo2025}).
Ray transforms have also been studied on non-compact non-hyperbolic Cartan--Hadamard manifolds~\cite{LRS2018}.

The geometric setting of our article lies somewhere in between standard Riemannian geometry ($\alpha = 0$) and asymptotically hyperbolic geometry ($\alpha = 2$), but in certain regards the gas giant geometry differs drastically from the cases~$\alpha = 0$ and $\alpha = 2$. For example, there are no constant curvature gas giants as was pointed out in~\cite{dHIKM2024} which implies that our geometry~$\alpha = 1$ is not simply an interpolation between the edge cases. The gas giant setting arose in~\cite{dHIKM2024} as the geometric setting for studying wave propagation on gas giant planets. Inverse problems, particularly the X-ray transform, has not been exhaustively studied in gas giant geometry, the only available result being the injectivity for the transform of scalar function~\cite{dHIKM2024}. The analysis of the eigenfunctions of the Laplace--Beltrami operator has seen more development (see~\cite{dVDdHT2024,Dietze2025,DR2024}).

A gas giant manifold is similar to a Grushin manifold~\cite{DR2024}.
Geodesics on Grushin manifolds have been studied as sub-Riemannian geodesics~\cite{ABG2025,Borza2022,BL2013,CC2009}.
This geometry arises also in quantum mechanics~\cite{BQ2023}.
The relations between the two smooth structures have also been explored in this setting~\cite{Romney2016}.
To the best of our knowledge our results on gas giant geodesics do not follow from the Grushin literature.

\subsection*{Acknowledgements}
J.I. was supported by the Research Council of Finland
(Flagship of Advanced Mathematics for Sensing Imaging and Modelling grant 359208; Centre of Excellence of Inverse Modelling and Imaging grant 353092; and other grants 351665, 351656, 358047, 360434) and a Väisälä project grant by the Finnish Academy of Science and Letters.
E.S. was supported by the Finnish Ministry of Education and Culture’s Pilot for Doctoral Programmes (Pilot project Mathematics of Sensing, Imaging and Modeling), the Research Council of Finland
(Flagship of Advanced Mathematics for Sensing Imaging and Modelling grant 359208; Centre of Excellence of Inverse Modelling and Imaging grant 353092), and the same Väisälä project grant as the first author.
A.K. was supported by the Geo-Mathematical Imaging Group at Rice University.

\section{Proof of the main result}
\label{sec:proof-of-the-main-result}

The setting of the proof is the geodesic flow on the cotangent bundle $T^*M$.
We mostly restrict the flow to the ``unit cosphere bundle'' $S^*M\subset T^*M$, which in gas giant geometry is a submanifold but not a subbundle up to the boundary.
The generator of this flow is the geodesic vector field~$X$.
Details can be found in section~\ref{sec:preliminaries}.

We prove the main result by using the Pestov identity~\eqref{eqn:Pestov full} on the transport equation $Xu^f=-\lambda f$ where $f$ is a 1-form with $If=0$ and the integral function $u^f\colon S^* M\to \R $ is the unique solution to the transport equation given by
\begin{equation}
\label{eqn:integral_function}
    u^f (z,\zeta)
    :=
    \int_0^{\tau (z,\zeta)} \ip{f(\gamma(t))}{\dot{\gamma}(t)}
    \der t
\end{equation}
for any lifted geodesic $(\gamma(t),\dot{\gamma}(t))\in TM$ given by the initial condition $(z,\zeta)\in S^* M$; see section~\ref{sec:preliminaries} below for more details. To use the Pestov identity on a gas giant we have to assume various regularity properties which we describe as a function space $\Omega$ as follows:

\begin{definition}[Function space $\Omega$]
\label{def:niceness-of-uf}
Let $u: S^* M\to \R$ be continuous. We say that $u\in \Omega$ if all the following hold:
\begin{enumerate}
\item
$u\in x^\infty C^\infty (S^* M^\circ)$
\item
$Xu\in C^\infty(S^* M)$
\item
$\nabla_{S^* M} u\in x^\infty L^\infty (S^* M, TS^* M)$
\item
$\gradv Xu, X\gradv u\in L^2 (N )$
\item
$u|_{\partial S^* M}=\gradv u|_{\partial S^* M}=0$
\end{enumerate}
See section~\ref{sec:preliminaries} for details on bundles, sections, and gradients.
\end{definition}

We prove the main theorem~\ref{thm:injectivity-mod-gauge} using the lemmas listed below.

\begin{lemma}[proved in section~\ref{sec:Boundary determination}]
\label{keylemma:bdy-det}
Let $(M,g)$ be a gas giant and let $f$ be a smooth 1-form on $M$. If $If=0$ then there exists $q\in C^\infty(M)$, supported near $\partial M$ with $q|_{\partial M}=0$, so that  $f-\der q\in x^\infty  C^\infty(M,T^* M)$.
\end{lemma}

\begin{lemma}[proved in section~\ref{sec:Regularity of the integral function}]
\label{keylemma:u-regularity}
Let $(M,g)$ be a non-trapping gas giant. If $If=0$ and $f\in  x^\infty C^\infty(M,T^* M)$ then $u^f\in \Omega$ in the sense of definition~\ref{def:niceness-of-uf}.
\end{lemma}

\begin{lemma}[proved in section~\ref{sec:Pestov identity}]
\label{keylemma:pestov}
Let $(M,g)$ be a non-trapping gas giant with non-positive curvature. If $u\in \Omega$ and $Xu(z,\zeta)$ is a polynomial of order~$1$ w.r.t.~$\zeta$, then there exists a function $p\in C^\infty(M)$ so that $u= \pi^* p$ on $S^* M$.
\end{lemma}

\begin{proof}[Proof of theorem~\ref{thm:injectivity-mod-gauge}]
We start by showing that $If=0$ implies $f=\der \Tilde{q}$ for some scalar function $\Tilde{q}$. Assume that $f$ is a smooth 1-form with $If=0$. Then by lemma~\ref{keylemma:bdy-det} we may use the lemma~\ref{keylemma:u-regularity} on the 1-form $\Tilde{f}:=f-\der q$. By lemma~\ref{keylemma:u-regularity} and the transport equation $Xu^{\Tilde{f}}=-\lambda \Tilde{f}$ giving an expression that is polynomial of order 1 w.r.t.\ $\zeta$ we may use lemma~\ref{keylemma:pestov} on the integral function $u^{\Tilde{f}}$ so that $u^{\Tilde{f}}(z,\zeta)=p(z)$ on $S^* M$. It follows that $p|_{\partial M}=0$. Since $u^{\Tilde{f}}$ is a solution of the transport equation $Xu^{\Tilde{f}}=-\lambda \Tilde{f}$ then we have
\begin{equation}
    Xu^{\Tilde{f}} =X\pi^* p=\lambda \der p=-\lambda \Tilde{f}
\end{equation}
Therefore $\Tilde{f}=-\der p$ and hence $f=\der (q-p)$ for $(q-p)|_{\partial M}=0$ which concludes the proof of the non-trivial direction of the claim. Note that from this same calculation we see that the function $p$ of lemma~\ref{keylemma:pestov} has to be smooth. If $p$ was not smooth then $\lambda\der p=X\pi^* p$ would also fail to be smooth which is a contradiction with the fact that $u^{\Tilde{f}}\in \Omega$.

To prove the other direction of the equivalence assume that $f=\der p$ for some $p\in C^\infty(M)$ with $p|_{\partial M}=0$. Then by $Xp=\lambda \der p$ and fundamental theorem of calculus we have $If=I\der p=0$.
\end{proof}

\subsection{Outline of the proofs of lemmas}

We prove the boundary determination lemma~\ref{keylemma:bdy-det} in section~\ref{sec:Boundary determination} where we first prove a result regarding short geodesics on gas giants in proposition~\ref{prop:short_geodesics}. We then use proposition~\ref{prop:short_geodesics} and a slightly modified version of the ideas of~\cite[proof of theorem 2.1]{LSU2003} to prove in proposition~\ref{prop:boundary determination} that on the boundary $\partial M$ the derivatives of a 1-form $f$ with $If=0$ vanish to the boundary normal direction up to a gauge described in lemma~\ref{lma:gauge correction gas giant}.
Lemma~\ref{keylemma:bdy-det} follows then from proposition~\ref{prop:boundary determination}.
Note that lemma~\ref{keylemma:bdy-det} only requires that $(M,g)$ is a gas giant so geometrically the result of this lemma exploits the fact that the boundary of a gas giant is (infinitely) strictly convex; see section~\ref{sec:analysis of exit time} for a more detailed discussion on strict convexity.

In section~\ref{sec:Regularity of the integral function} we prove lemma~\ref{keylemma:u-regularity} by proving a slightly weaker regularity result for the integral function $u^f$ in proposition~\ref{prop:blow-up rates of u}. We then show that lemma~\ref{keylemma:u-regularity} follows from proposition~\ref{prop:blow-up rates of u} and throughout section~\ref{sec:Regularity of the integral function} we prove various regularity results which we then compile to prove proposition~\ref{prop:blow-up rates of u} at the end of section~\ref{sec:Regularity of the integral function}. In lemma~\ref{keylemma:u-regularity} we have to assume that the manifold is globally non-trapping since the integral function is not well-defined without this assumption. 
This feature of the global geodesic flow does not follow from gas giant geometry.

Finally in section~\ref{sec:Pestov identity} we prove a Pestov identity for functions in the space $\Omega$. The conclusion of using the Pestov identity is in corollary~\ref{cor:corollary of Pestov} and lemma~\ref{lma:Pestov conclusion}. Then lemma~\ref{keylemma:pestov} follows directly from lemma~\ref{lma:Pestov conclusion}. To simplify the proof of lemma~\ref{keylemma:pestov} we assume that the manifold $(M,g)$ has non-positive curvature. We expect that this assumption may be replaced with the assumption that the manifold has no conjugate points, so that $(M,g)$ would then be a simple gas giant, but we choose not to pursue this level of generality due to technical difficulties caused by curvature blowing up at the boundary.

\section{Preliminaries}
\label{sec:preliminaries}

We describe the mathematical setup of our result here. Let $(M,\overline{g})$ be an $n$-dimensional smooth compact Riemannian manifold with boundary, for $n\ge 2$. As in  section~\ref{sec:intro} we equip $M$ with a gas giant metric $g=\rho^{-1}\overline{g}$ where $\rho$ is a boundary defining function and call the pair $(M,g)$ a gas giant. We denote the interior of $M$ by $M^\circ$.  Let $U\subset \partial M$ be a neighborhood of the boundary and let $x\in [0,1)$. Then near the boundary of a gas giant there exists a smooth structure $(x,y)\in [0,1)\times \partial M$ and a family of smooth metrics $h_x$ varying smoothly w.r.t.\ $x$ so that
\begin{equation}
g=\der x^2+x^{-2}h(x,y,\der y)
\end{equation}
in this smooth structure. As seen in the introduction (section~\ref{sec:intro}) we can scale the coordinate $x$ to $\rho$ and use the result~\cite[proposition 2]{dHIKM2024} to show that this smooth structure exists. All smoothness analysis up to the boundary is done w.r.t.\ the smooth structure $(x,y)$.

We use the coordinates $z=(x,y)$ with the convention that $z^0=x$ and $z^i=y^i$ for $i\geq1$.
Greek indices (e.g. $\mu,\nu$) run from $0$ to $n-1$ and Latin ones (e.g. $i,j$) from $1$ to $n-1$.
We use the Einstein summation convention.
The Christoffel symbols of the metric $g=\der x^2+x^{-2}h(x,y,\der y)$ are given by
\begin{align}
    &\Gamma^0_{00}=\Gamma^i_{00}=\Gamma^0_{i0}=0,\\
    &\Gamma^i_{j0}=\frac{1}{2}h^{ik}\partial_x h_{kj}-\delta^i_{\ j} x^{-1},\\
    &\Gamma^0_{ij}=x^{-3}h_{ij}-\frac{1}{2}x^{-2} \partial_x h_{ij},\\
    &\Gamma^i_{jk}=H^i_{jk}
\end{align}
where $H^i_{jk}$ are the Christoffel symbols of $h$. By~\cite[proposition 1]{dHIKM2024} we have $R\sim x^{-2}$ in the sectional curvature sense for a gas giant so that the curvature approaches $-\infty$ up to the boundary. The volume form of the gas giant is given by $\der V_g= x^{-(n-1)}\der x\, \der V_h$ near the boundary. 
Therefore the volume $\Vol_g(M)$ of the gas giant is infinite.

Next we will set the notion of geodesics on a gas giant as a Hamiltonian flow. Let $(z,\zeta)\in T^* M$, with $z=(x,y)$ and $\zeta=(\xi,\eta)$, and let $H(x,y,\xi,\eta):=\frac{1}{2}\xi ^2 +\frac{1}{2}x^2 h^{ij}\eta_i\eta_j$ be the Hamiltonian. The geodesic flow $\phi_t:T^* M\to T^* M$ is given by the initial condition $(z,\zeta)$ and Hamilton's equations
\begin{align}
        \dot{x}&=\frac{\partial H}{\partial \xi}=\xi,
        \\
        \dot{y}^i&=\frac{\partial H}{\partial \eta_i}=x^2 h^{ij}\eta_j, \\
        \dot{\xi}&=-\frac{\partial H}{\partial x}=-x h^{ij}\eta_i\eta_j-\frac{1}{2}x^2\frac{\partial h^{ij}}{\partial x}\eta_i\eta_j,\\
        \dot{\eta}_i&=-\frac{\partial H}{\partial y^i}=-\frac{1}{2}x^2 \frac{\partial h^{kj}}{\partial y^i}\eta_k\eta_j
\end{align}
so that $\phi_t(z,\zeta)=(x(t),y(t),\xi(t),\eta(t))$. We denote the exit time of a geodesic by $\tau(z,\zeta)$ so that $t\in [0,\tau(z,\zeta)]$. The Hamiltonian vector field of the geodesic flow is given by
\begin{equation}
X=\xi\partial_x +x^2h^{ij}\eta_j\partial_{y^i}-\left ( xh^{ij}\eta_i\eta_j-\frac{1}{2}x^2\frac{\partial h^{ij}}{\partial x}\eta_i\eta_j \right )\partial_\xi -\frac{1}{2}x^2\frac{\partial h^{kj}}{\partial y^i}\eta_k\eta_j\partial_{\eta_i}
\end{equation}
We will analyze the boundary behavior of the geodesic flow in more detail in sections~\ref{sec:analysis of exit time} and~\ref{sec:Boundary determination}. By using the canonical projection $\pi\colon T^* M\to M$ we define the geodesic with an initial condition $(z,\zeta)\in T^* M$ by $\gamma_t:=\pi (\phi_t(z,\zeta))$ for which there exists a lift to $TM$ by $(\gamma_t,\dot{\gamma}_t):= (z(t)),\zeta^\sharp (t))$, where $\zeta^\sharp =(\xi , x^2h^{ij}\eta_j)$ near the boundary. 

Let $S^* M=\{(z,\zeta)\in T^* M\colon \abs{\zeta}_g=1\}$.
Then the identification $\lambda$ of a 1-form $f$ to a function on $S^* M$ is given by
\begin{equation}
\lambda f(z,\zeta):=f_\mu(z) g^{\mu\nu}(z)\zeta_\nu
\end{equation}
so that the integral function $u^f\colon S^* M\to \R $ is defined by
\begin{equation}
u^f(z,\zeta):=\int_0^{\tau(z,\zeta)}f_\mu (z(t))g^{\mu\nu}(z(t))\zeta_\nu(t)\der t=\int_0^{\tau(z,\zeta)}f_\mu (\gamma(t))\dot{\gamma}^\mu (t) \der t.
\end{equation}
By setting $\partial_\pm S^* M=\{(0,y,\xi,\eta)\in \partial T^* M\colon \xi=\pm 1\}$ we have the restriction $u^f|_{\partial_+ S^* M}=If$ and the disjoint union $\partial S^* M =\partial_+ S^*M\sqcup \partial_- S^* M$. In a similar fashion to $\partial_\pm S^* M$ we denote $\partial_0 T^* M \cong T^* \partial M$ by $\partial_0T^* M=\{ (0,y,\xi,\eta)\in \partial T^* M\colon \xi =0\}$.

Next we will give a standard collection of results regarding the geometry of the unit cosphere bundle $S^* M$ that we will use freely in sections~\ref{sec:Regularity of the integral function} and~\ref{sec:Pestov identity}. For details cf.~\cite{Pat1999} and for local coordinate formulas of various objects described below see appendix~\ref{sec:appendix}.

For the remainder of this section let $g$ be a Riemannian metric smooth up to the boundary and denote by $\pi$ the restriction of the canonical projection of $T^* M$ to $S^* M$. Let $(z,\zeta)\in S^* M$, $\theta\in T_{(z,\zeta)}S^* M$ and $c$ be a curve in $S^* M$ with $c(0)=(z,\zeta)$ and $\dot{c}(0)=\theta$. Then the connection map $K\colon TS^* M\to TM$ is defined by $K(\theta)=\Der_t c (0)^\sharp$ where $\sharp$ denotes the musical isomorphism $T^* M\to TM$ and $\Der_t$ is the Levi-Civita connection along the curve $\pi (c(t))$ on $M$. Let $\mathcal{V}:=\ker \der \pi$ denote the vertical bundle and $\mathcal{H}:=\ker K$ the horizontal bundle. There is a splitting
\begin{equation}
    TS^* M=\R X \oplus \mathcal{H}\oplus \mathcal{V}
\end{equation}
which is orthogonal w.r.t.\ the Sasaki metric of $S^* M$ (see equation~\eqref{eqn:Sasaki metric}). The maps $\der \pi |_{\mathcal{H}}\colon\mathcal{H}\to N$ and $K|_{\mathcal{V}}\colon \mathcal{V}\to N$ are isomorphisms where $N\to S^* M$ is the normal bundle defined in equation~\eqref{eqn:normal bundle}. The Sasaki gradient of a smooth function $u\in C^\infty (S^* M)$ is given by
\begin{equation}
\nabla_{S^* M} u
=
(Xu)X
+
\gradv u
+
\gradh u
\end{equation}
where $X$ is the Hamiltonian vector field of the geodesic flow and $\gradv u, \gradh u$ are the vertical and horizontal gradients which are smooth sections of the normal bundle $N$, see equations~\eqref{eqn:vertical gradient local} and~\eqref{eqn:horizontal gradient local} for coordinate expressions of the gradients. For a section $W$ of the normal bundle $N$ we can view the Riemannian curvature tensor as a mapping of $W$ to another section of $N$ via $RW(z,\zeta)= R(W(z,\zeta),\zeta^\sharp)\zeta^\sharp$. We denote the volume form induced by the Sasaki metric on $S^* M$ by $\der \Sigma$.
Following the convention of~\cite[Proof of theorem 1.1]{GGSU2019}, the volume form of $\partial S^* M$ is given by $\der \sigma=j^* i_X \der \Sigma$, where $j\colon\partial S^* M\to S^* M$ is the inclusion and $i_X$ denotes the interior product w.r.t. the Hamiltonian vector field $X$.

\section{Analysis of the exit time}
\label{sec:analysis of exit time}

This section contains the analysis of the geodesic flow on the cotangent bundle when we choose the smooth structure to be such that $g=\der x^2+x^{-2}h(x,y,\der y)$. As a result of the analysis we prove proposition~\ref{prop:asymptotic of exit time} regarding the asymptotic behavior of the exit time of geodesics on a gas giant. Most of the arguments in this section are very similar to those in~\cite[section 2.2]{dHIKM2024} but since we are using a different smooth structure we will list some properties of the flow that will be essential in proving the main result.

To clear up notation and to give geometric intuition to the analysis of the geodesic flow we introduce the notion of second fundamental form of the boundary that is extended to the neighborhood where the boundary normal coordinates are defined by the following

\begin{proposition}
    For the metric $g = \der x^2+x^{-2}h(x,y,\der y)$ and the tangent vectors $a=\omega^\sharp$, $b=\theta
     ^\sharp$, where $\theta,\omega\in T^* M^\circ$, we have the second fundamental form of the level sets of $x$ and its dual
     \begin{align}
         I\!I(a,b)&:= x^{-3}h_{ij}a^ib^j-\frac{1}{2}x^{-2}\partial_x h_{ij}a^ib^j\\
         I\!I(\omega,\theta)&:=x h^{ij}\omega_i\theta_j +\frac{1}{2}x^2   \partial_x h^{ij} \omega_i\theta_j
     \end{align}
     in the sense that $I\!I(a,b)=I\!I(\omega,\theta)$ up to musical isomorphisms.
\end{proposition}

\begin{proof}
    The claim follows from a calculation by using the two formulas
    \begin{equation}
     I\!I(a,b)
     =
     -\frac{1}{2}\partial_x g_{ij}a^ib^j \teksti{and}
     x^2 \partial_x h^{ij}
     =
     -x^2 h^{im}h^{kj}\partial_x h_{mk}.\qedhere
    \end{equation}
\end{proof}

We can then write
\begin{equation}
\dot{\xi}(t)=-I\! I(\eta,\eta)
\end{equation}
and Taylor expanding $h^{ij}(x,y)$ and $\partial_x h^{ij}(x,y)$ in $x$ at $x=0$ gives us
\begin{equation}
    \dot{\xi} (t)=-I\! I(\eta,\eta)=- x h^{ij}(0,y)\eta_i\eta_j+\Order(x^2)
\end{equation}
Since $h(0,y)$ is just the metric on $\partial M$ then $h^{ij}(0,y)\eta_i\eta_j>0$ when $\eta\not=0$. Hence for some small $\varepsilon>0$ there is a region defined by $0<x<\varepsilon$ where $\dot{\xi}\le 0$. By the same argument $I\! I(\eta,\eta)>0$ in this region so the boundary of a gas giant is strictly convex in some sense.

Fix any $0<x_0<\varepsilon$ in this region and some reference point $(y,\eta)\in T^* \partial M$ with $\eta\not=0$. Then the above argument and $\dot{x}(t)=\xi(t)$ say that for $\xi_0\le 0$ there exists an exit time $\tau(x_0,y,\xi_0,\eta)$ for which $x(\tau)=0$. To simplify the rest of the argument we will use the fact proved in~\cite{dHIKM2024} that $\tau (x_0,y,\xi_0,\eta)<\infty$. Since $\xi(t)<0$ for $t\in (0,\tau]$ then $x$ is strictly decreasing and we may use $x$ as a variable. Next observe that
\begin{equation}
\frac{\der \xi}{\der x}
=
\frac{\der \xi/\der t}{\der x/\der t}
=
-\frac{I\! I(\eta,\eta)}{\xi}.
\end{equation}
Thus
\begin{equation}
\frac{\der }{\der x}\xi(x)^2
=
-2I\! I(\eta,\eta)
=
-\partial_x (x^2 h^{ij}\eta_i\eta_j).
\end{equation}
Integrating the last expression w.r.t.\ $x$ from 0 to $x_0$ gives us
\begin{equation}
    \xi(x_0)^2-\xi(0)^2
    =
    -x_0^2 h^{ij}(x_0,y)\eta_i\eta_j
\end{equation}    
and thus
\begin{equation}
    \xi(x_0)^2 +x_0^2 h^{ij}(x_0,y)\eta_i \eta_j
    =
    \xi(0)^2.
\end{equation}
Since the Hamiltonian $H(x,y,\xi,\eta)$ is preserved along the geodesic flow, then by reversing the flow we have that the last expression says that the speed of a given maximal geodesic is characterized by the normal component $\xi>0$ of the initial condition $(0,y,\xi,\eta)\in \partial T^* M$. This also means that any initial condition $(0,y,0,\eta)$ in the codirection of the boundary will result in a geodesic with zero speed, i.e. all geodesics with initial condition cotangent to the boundary are stationary. We will explore the general case of arbitrary speeds and the zero speed limit in the boundary determination section~\ref{sec:Boundary determination} below but for now we will focus on the unit speed case $\xi(0)=1$. 

Let us go back to the geodesic where $x$ is strictly decreasing so we may use it as a variable. Then by using the Taylor expansion $(a-b)^{-1/2}=a^{-1/2}+\Order (a^{-3/2}b)$ and the fact that $\xi\le 0$ we get
\begin{equation}
\xi(x_0)^2=1-x_0^2h^{ij}(x_0,y)\eta_i\eta_j 
\end{equation}
so that
\begin{equation}
\label{eqn:asymptotic of xi}
\begin{split}
    -\frac{1}{\xi(x_0)}
    &=
    \left ( 1-x_0^2h^{ij}(x_0,y)\eta_i\eta_j \right )^{-1/2}
    \\
    &=
    1+x_0^2h^{ij}(x_0,y)\eta_i\eta_jF(x_0,y,\eta)
\end{split}
\end{equation}
where $F$ is an auxiliary function that is uniformly bounded in $x_0,y,\eta$. With this discussion we are ready to prove the following. 

\begin{proposition}\label{prop:asymptotic of exit time}
Let $(M,g)$ be a gas giant. Fix some small $\varepsilon>0$ and let $U\subset \partial M$ be some neighborhood where normal coordinates $(x,y)\in [0,\varepsilon)\times U$ are defined and we have $I\! I(\eta,\eta)>0$. Fix $x_0\in (0,\varepsilon)$ and $y\in U$, and let $(\xi_0,\eta)\in S^*_{(x_0,y)}M$ with $\xi_0\le 0$. Then we have the estimate
    \begin{equation}
    \tau(x_0,y,\xi_0,\eta)=x_0+\Order (x_0^{3})
    \end{equation}
where the remainder terms are uniformly bounded in $(y,\xi_0,\eta)$. Namely $\tau\to 0$ uniformly in $(y,\xi_0,\eta)$, with $\xi_0\le 0$, as $x_0\to 0$.
\end{proposition}

\begin{proof}
    Let $\varepsilon>0$ be small enough so that $\xi_0\le 0$ implies $\xi(t)<0$ for $t\in (0,\tau(x_0,y,\xi_0,\eta)]$ when $0<x_0<\varepsilon$. Then $x(t)$ is strictly decreasing along the flow so we may use $x$ as a change of variables with $\der x=\xi\der t$. We will use the fact that $h^{ij}\eta_i\eta_j$ is a uniformly bounded quantity along the given geodesic, cf.\ lemma~\ref{lma:h-norm bound} and its proof below for the verification of this fact. Using the asymptotic of $-1/\xi$ in equation~\eqref{eqn:asymptotic of xi} we get
\begin{equation}
\begin{split}
    \tau (x_0,y,\xi_0,\eta)
    &=
    \int_0^{\tau (x_0,y,\xi_0,\eta)}\der t
    \\
    &=
    \int_0^{x_0}-\frac{\der x}{\xi}
    \\
    &=
    \int_0^{x_0}\left ( 1+x^2h^{ij}(x,y)\eta_i\eta_jF(x,y,\eta) \right ) \der x
    \\
    &=
    x_0+\Order (x_0^3)
\end{split}
\end{equation}
    which proves the claim.
\end{proof}

\section{Boundary determination}
\label{sec:Boundary determination}

\subsection{Existence of short geodesics}

In this section the goal is to prove that given any fixed point $y\in \partial M$ and a boundary cotangent direction $\eta\in T^* _y\partial M$ there exists a family of maximal geodesics so that the flow on the cotangent bundle corresponding to the geodesics in this family stays in an arbitrarily small neighborhood $U\subset T^* M$ of the point $(0,y,0,\eta)\in \partial T^* M$. In particular the length of the maximal geodesics in this family will become arbitrarily small hence the name short geodesics.

\begin{proposition}
\label{prop:short_geodesics}
Let $(M,g)$ be a gas giant and $\varepsilon>0$ be small. Fix $y\in \partial M$. For $0<s <\varepsilon$ we take the one parameter set of initial conditions of geodesics $\sigma(s)=(s,\eta)\in T^* _{(0,y)} M$ where $\eta\in T^* _y\partial M$ is independent of $s$ with $\abs{\eta}_h^2=1$. Let $\gamma_s\colon [0,\tau(0,y,\sigma_s)]\to T^* M$ denote the geodesic corresponding to the initial condition $\sigma_s$ and let $0<l(\gamma_s) $ denote the length of the geodesic. Then $\gamma_s(t)\to (0,y,0,\eta)$ as $s\to 0$ uniformly in $t\in [0,\tau(0,y,\sigma_s)]$ and $l(\gamma_s)\to 0$ as $s\to 0$.
\end{proposition}

We prove the claim of proposition~\ref{prop:short_geodesics} with the following three lemmas and after the proof of~\ref{prop:short_geodesics} we will proceed with proving the lemmas.

\begin{lemma}
\label{lma:height estimate}
    Let $\varepsilon>0$ be small and $0<s<\varepsilon$. Then for a geodesic $\gamma_s(t)=(z(t),\zeta(t))$ with the initial condition $(0,y,\sigma_s)$ we have $\dot{\xi}(t)<0$ for $t\in (0,\tau(0,y,\sigma_s))$ and
    \begin{equation}
    x(t) \le Cs\teksti{for} t\in [0,\tau(0,y,\sigma_s)]
    \end{equation}
    where the constant is uniform in $t$ and depends only on the metric $h$.
\end{lemma}

\begin{lemma}
\label{lma:time estimate}
Let $(0,y,\sigma_s)$ be the initial condition of a geodesic as above. Then we have the uniform estimate for the exit time
\begin{equation}
0<\tau(0,y,\sigma_s)\le C
\end{equation}
and for the length of the geodesic $\gamma_s$
\begin{equation}
0<l(\gamma_s)=s\tau(0,y,\sigma_s)\le Cs
\end{equation}
where the constant $C$ only depends on the metric $h$.
\end{lemma}

\begin{lemma}
\label{lma:h-norm bound}
Let us denote $h^{ij}(x,y)\eta_i\eta_j:=h(\eta,\eta)$ and $h^{ij}(0,y)\eta_i\eta_j:=h_0(\eta,\eta)$. Fix $h_0(\eta,\eta)=1$ at the initial condition of the geodesic $\gamma_s$. Then there exists a constant $c>0$ which depends only on the metric $h$ so that $\frac{1}{2}\le h(\eta(t),\eta(t))\le \frac{3}{2}$ along the geodesic $\gamma_s$ whenever $s\le c$.
\end{lemma}

From the lemmas we observe some immediate geometric facts about the short geodesics $\gamma_s$. For $0<s<1$ the projections $\pi(\gamma_s(t))\in M$ are pregeodesics in the sense that $\gamma_s$ are solutions of Hamilton's equations but the paths $\pi(\gamma_s(t))$ are not isometric embeddings of the parameter interval onto the manifold. If one wishes to change the mapping to the isometric embedding $\gamma_s\colon[0,l(\gamma_s)]\to M$ then from the length formula in lemma~\ref{lma:time estimate}, which is an elementary formula of type ''distance = speed $\times$ time'', we see that one has to scale $s\mapsto 1$. Since we can think of $T^*_{(0,y)}M$ as a vector space then this reparametrization corresponds to mapping the initial condition with the gnomonic projection $\sigma(s)\mapsto \sigma(s)/s=(1,\eta/s)\in S^*_{(0,y)} M$. It is more popular to state and use results similar to proposition~\ref{prop:short_geodesics} in the unit (co)sphere bundle~\cite{dHIKM2024,GGSU2019,LSU2003,PSU2023} due to the analysis of geodesics being more straight-forward but now we see that on gas giants in the limit $s\to 0$ the reparametrized initial condition $(1,\eta/s)$ on $S^* M$ will blow up and so on gas giants it is more natural to change the speed of geodesics instead of analyzing the unit speed geodesics. We proceed by proving the short geodesic result of proposition~\ref{prop:short_geodesics}.

\begin{proof}[Proof of proposition \ref{prop:short_geodesics}]
    Let $\varepsilon>0$ be small enough so that the all of the lemmas~\ref{lma:height estimate},~\ref{lma:time estimate} and~\ref{lma:h-norm bound} hold. Throughout the proof we denote any uniform constant by $C>0$. Since the $h$-metric is smooth then by lemmas~\ref{lma:height estimate},~\ref{lma:h-norm bound} and Hamilton's equations we get for all $i=1,\dots ,n-1$ that
    \begin{equation}
    \abs{\dot{y}^i(t)}\le C x(t)^2 \le Cs^2\teksti{and}\abs{\dot{\eta}_i(t)}\le Cs^2.
    \end{equation}
    Hence by mean value theorem and lemma~\ref{lma:time estimate}
    \begin{align}
        &\abs{y^i(t)-y^i(0)}\le Cs^2t\le Cs^2\tau\le Cs^{2}\\
        &\abs{\eta_i(t)-\eta_i(0)}\le Cs^2t\le Cs^2\tau\le Cs^{2}
    \end{align}
    for all $t\in [0,\tau(0,y,\sigma_s)]$ and $i=1,\dots ,n-1$. By lemma~\ref{lma:height estimate} we know that $\xi(t)$ is strictly decreasing along the geodesic. By combining this with the constant speed condition of the geodesic $\gamma_s$ we get $-s\le \xi(t)\le s$ so
    \begin{equation}
    \abs{\xi(t)-\xi(0)}\le 2s
    \end{equation}
    for all $t\in [0,\tau(0,y,\sigma_s)]$. By lemma~\ref{lma:height estimate} we also have
    \begin{equation}
    x(t)=\abs{x(t)-x(0)}\le Cs
    \end{equation}
    for all $t\in [0,\tau(0,y,\sigma_s)]$. Hence taking the limit $s\to 0$ and using the uniform estimates above give us the claim of uniform convergence. From lemma~\ref{lma:time estimate} we get $l(\gamma_s)\to 0$ as $s\to 0$.
\end{proof}

We will now prove lemmas~\ref{lma:height estimate} and~\ref{lma:h-norm bound} by analyzing the flow for small enough $s$. From Hamilton's equations we see that if we choose $s$ small enough then the geodesic stays in a region where $\dot{\xi}(t)<0$ for $t\in (0,\tau(0,y,\sigma_s))$. We will give a more precise argument for this property below. Whenever the geodesic stays in such a region there exists a time $T\in (0,\tau(0,y,\sigma_s))$ so that $0<x(t)\le x(T)$ for $t\in (0,\tau(0,y,\sigma_s))$ and $\xi(T)=0$. In particular $0<T<\tau(0,y,\sigma_s)$ so each of the geodesics $\gamma_s$ has positive length. The goal is to analyze the constant speed condition at this time $T$ and to do that we need to estimate the $h$-norm. We prove the following useful lemma.
\begin{lemma}
\label{lma:constant of motion}
    Let $h(\eta,\eta)=h^{ij}(x,y)\eta_i\eta_j$ and $\eta\not=0$. We have $Xh(\eta,\eta)=0$ along a geodesic near the boundary if and only if $\partial_x h=0$. If this condition holds then the $h$-norm is a constant of motion for the geodesic flow.
\end{lemma}

\begin{proof}
    By Liouville's theorem
    \begin{equation}
    \begin{split}
        \frac{\der }{\der t} h(\eta,\eta)
        &=
        Xh(\eta,\eta)
        \\
        &=
        \xi \partial_xh^{ij}\eta_i\eta_j
        +
        x^2 h^{ij}\eta_j\partial_{y^i}h^{ml}\eta_m\eta_l
        \\
        &\quad
        -
        \frac{1}{2}x^2\partial_{y^i}h^{kj} \eta_k\eta_jh^{ml}(\delta^i_{\ m}\eta_l+\eta_m\delta^i_{\ l})\\
        &=\xi\partial_x h^{ij}\eta_i\eta_j. \qedhere
    \end{split}
    \end{equation}
\end{proof}

From the lemma and its proof we observe the immediate corollary regarding the boundary $h_0$-norm.
\begin{corollary}
    Let $h_0(\eta,\eta)=h^{ij}(0,y)\eta_i\eta_j$. Then $h_0(\eta,\eta)$ is a constant of motion for the geodesic flow.
\end{corollary}

This means that if we fix the initial condition $(0,y,s,\eta)$ with $h(\eta,\eta)=h_0(\eta,\eta)=1$ then $h_0(\eta(t),\eta(t))\equiv 1$ along the geodesic. This also means that the components of $\eta$ have to be bounded uniformly along the flow since $h_0$ is smooth. It follows that if we reparametrize the initial condition $(s,\eta)\mapsto (1,\eta/s)$ of the short geodesic $\gamma_s$ of proposition~\ref{prop:short_geodesics} then along the unit speed version of the geodesic flow on the cotangent bundle we have 
\begin{equation}
\frac{\der \xi(\Tilde{t})}{\der \Tilde{t}}=-I\! I \big (\eta(\Tilde{t})/s,\eta(\Tilde{t})/s\big)=-I\! I\big (\eta(\Tilde{t}),\eta(\Tilde{t})\big )/s^2
\end{equation}
where $\eta( \Tilde{t})$ is bounded along the flow and $\Tilde{t}=st$ is the parameter along the unit speed version of the geodesic $\gamma_s$. By choosing $s>0$ small we see that $\xi(\Tilde{t})$ decreases faster near the boundary where $I\! I(\eta,\eta)>0$ and hence for small enough $s$ the distance from the boundary $x(\Tilde{t})$ reaches its maximum in a region where $\der \xi (\Tilde{t})/\der \Tilde{t}<0$. Since $\der \xi (\Tilde{t})/\der \Tilde{t}<0$ holds along the unit speed version of $\gamma_s$ then also $\dot{\xi}< 0$ has to hold in the case where we have speed $s$ along the geodesic $\gamma_s$ whenever we choose $s$ small enough. With this discussion we are now ready to prove lemma~\ref{lma:height estimate}.

\begin{proof}[Proof of lemma~\ref{lma:height estimate}]
Let $\varepsilon>0$ be small enough so that for $0<s<\varepsilon$ we have $\dot{\xi}(t)<0$ for $t\in (0,\tau (0,y,\sigma_s))$. Fix $h(\eta,\eta)=1$ for the initial condition. Then we have a time $T$ for which $x(T)$ is maximum of $x$ along the flow and $\xi(T)=0$. Taylor approximating the functions $h^{ij}(x(T),y(T))$ w.r.t.\ $x(T)$ at $x(T)=0$ gives the following form for the constant speed condition of the geodesic
\begin{equation}
\begin{split}
    0
    &=
    \xi(T)^2
    \\
    &=
    s^2-x(T)^2 h(\eta(T),\eta(T))
    \\
    &=
    s^2-x(T)^2 h_0(\eta(T),\eta(T))+\Order(x(T)^{3})
    \\
    &=
    s^2-x(T)^2 +\Order(x(T)^{3})
\end{split}
\end{equation}
By the smoothness of $h$ and boundedness of $\eta$ the remainder term $\Order(x(T)^{3})$ is bounded with a uniform constant that only depends on the metric $h$. We see that $x(T)^2$ is comparable to $s^2$. By iteration on the remainder terms we get
\begin{equation}
    x(T)^2 =s^2+\Order(s^{3})\le s^2 +Cs^{3}=s^2(1+Cs)\le s^2(1+C)
\end{equation}
where $C>0$ denotes any uniform constant that only depends on the metric $h$ and we used an arbitrary bound $s\le 1$. 
\end{proof}
With the same ideas we find the proof of lemma~\ref{lma:h-norm bound}

\begin{proof}[Proof of lemma~\ref{lma:h-norm bound}]
Fix $h_0(\eta,\eta)\equiv 1$. Then by Taylor expanding at any point along the geodesic flow we have
\begin{equation}
    h(\eta,\eta)=h_0(\eta,\eta)+xP(x,y,\eta)=1+xP(x,y,\eta)
\end{equation}
where $\abs{P(x,y,\eta)}\le C_1$ so that $C_1>0$ is a uniform constant that only depends on the metric $h$. By lemma~\ref{lma:height estimate} we have $x\le C_2s$ for $C_2>0$ so that
\begin{equation}
    h(\eta,\eta)\le 1+C_1C_2 s\le \frac{3}{2}
\end{equation}
holds when
\begin{equation}
s\le \frac{1}{2C_2 C_1}
\end{equation}  
A similar argument gives the lower bound for $h(\eta(t),\eta(t))$ in the claim.
\end{proof}

Finally we will prove lemma~\ref{lma:time estimate} by making the proof of proposition~\ref{prop:asymptotic of exit time} more precise and using the lemmas~\ref{lma:height estimate} and~\ref{lma:h-norm bound}.

\begin{proof}[Proof of lemma~\ref{lma:time estimate}]
Let $\gamma_s$ be the geodesic with the initial condition $(0,y,\sigma_s)$. Then there exists a time $T$ for which $\xi(T)=0$. We may identify the geodesic $\gamma_s$ as the combination of two geodesics $\gamma_\pm$ so that $\gamma_+$ is the geodesic with the initial condition $(x(T),y(T),0,\eta(T))$ and $\gamma_-$ is the geodesic with the initial condition $(x(T),y(T),0,-\eta(T))$. Denote the exit time of $\gamma_+$ by $\tau_+$ and the exit time of $\gamma_-$ by $\tau_-$. Then $\tau(0,y,\sigma_s)=\tau_+ +\tau_-$ and we use the integral from the proof of proposition~\ref{prop:asymptotic of exit time} to estimate $\tau_\pm$. Note that along the geodesics $\gamma_\pm$ we have $\xi<0$ for small $s>0$ so we may change the integration variable in the same way as in the proof of proposition~\ref{prop:asymptotic of exit time}. We have
\begin{equation}
\begin{split}
    \tau_\pm
    &=
    \int_0^{\tau_\pm}\der t
    \\
    &=
    \int_0^{x(T)} -\frac{1}{\xi}\der x
    \\
    &=
    \int_0^{x(T)} \frac{1}{\sqrt{s^2-x^2 h(\eta,\eta)}}\der x
    \\
    &=
    \int_0^{x(T)}\frac{1}{s\sqrt{1-\frac{x^2}{s^2}h(\eta,\eta)}}\der x
\end{split}
\end{equation}
Observe that in the last integral the function $x^2 h(\eta,\eta)=s^2-\xi^2$ is strictly increasing w.r.t. $x$ from $x=0$ to $x=x(T)$ since by choosing $s$ small we know that $\xi$ is strictly decreasing along this interval. Hence we may introduce the following change of variables
\begin{equation}
    u^2=\frac{x^2}{s^2}h(\eta,\eta)\teksti{with} \der u^2=2u\der u=\frac{2}{s^2}I\! I(\eta,\eta)\der x
\end{equation}
so
\begin{equation}
    \der x=\frac{s^2 u}{I\! I(\eta,\eta)}\der u=\frac{2s xh(\eta,\eta)^{1/2}}{2xh(\eta,\eta)+x^2\partial_x h(\eta,\eta)}\der u.
\end{equation}
For $x>0$ and $s$ small enough we have by lemmas~\ref{lma:height estimate} and~\ref{lma:h-norm bound} that also $x$ is small and $h(\eta,\eta)$ is bounded and strictly positive so
\begin{equation}
    \frac{2 xh(\eta,\eta)^{1/2}}{2xh(\eta,\eta)+x^2\partial_x h(\eta,\eta)}=\frac{2h(\eta,\eta)^{1/2}}{2h(\eta,\eta)+x\partial_x h(\eta,\eta)}\le C_h
\end{equation}
where $C_h$ is a uniform constant that only depends on the metric $h$. Therefore the change of variables is well-defined and we obtain
\begin{equation}
\begin{split}
    \tau_\pm
    &=
    \int_0^{x(T)}\frac{1}{s\sqrt{1-\frac{x^2}{s^2}h(\eta,\eta)}}\der x
    \\
    &\le
    \int_0^1 \frac{1}{\sqrt{1-u^2}}C_h\der u
    \\
    &=
    C_h\arcsin (1)
    \\
    &=
    \frac{C_h\pi}{2}
\end{split}
\end{equation}
and by $\tau=\tau_+ + \tau_-$ we get a similar estimate for the exit time of $\gamma_s$. The estimate for the length of the geodesic follows from
\begin{equation}
l(\gamma_s)
=
\int_0^{\tau(0,y,\sigma_s)}s \der t
=
s\tau(0,y,\sigma_s).
\qedhere
\end{equation}
\end{proof}

Observe that from the proof we have the stronger result that if the $h$-norm $h(\eta,\eta)$ is a constant of motion for the geodesic flow as in lemma~\ref{lma:constant of motion}, which holds when e.g.\ the background metric $\overline{g}$ of the gas giant is Euclidean, then lemma~\ref{lma:time estimate} can be improved to say that $\tau(0,y,\sigma_s)=\pi$ and the length of the short geodesics are given by $l(\gamma_s)=s\pi$. In the case when $\overline{g}$ is Euclidean and the dimension of the gas giant is two one can solve Hamilton's equations as in~\cite[chapter 11]{CC2009} and recover these same observations about the exit time $\tau=\pi$ and the length of the geodesics $l(\gamma_s)=s\pi$ so we have essentially generalized these observations in lemma~\ref{lma:time estimate}.

\subsection{Boundary determination of 1-forms}

Here we prove a boundary determination result for smooth 1-forms which have a vanishing geodesic X-ray transform over maximal geodesics. We will then use this result to prove lemma~\ref{keylemma:bdy-det}. We start by applying a gauge correction to the smooth 1-form as in the following lemma.

\begin{lemma} [Gauge correction]
\label{lma:gauge correction gas giant}
     Let $(M,g)$ be a gas giant.
     Let $f$ be a smooth 1-form with $If =0$. Then near the boundary $\partial M$ there exists a scalar function $q$ so that $q|_{\partial M}=0$ and for $\Tilde{f}:=f-\der q$ we have $I\Tilde{f}=0$ and $\Tilde{f}_0\equiv 0$ in any collar neighborhood of the boundary.
\end{lemma}

\begin{proof}
    We will use similar ideas as in~\cite[proof of lemma 4]{SU2005}. For a neighborhood $U\subset \partial M$ let 
    \begin{align}
        q(x,y)=\int_0^x f_0(s,y)\der s\teksti{for} (x,y)\in [0,1)\times U
    \end{align}
    By smoothness of $f$ we have $q(0,y)\equiv 0$. Clearly $\lambda \der q= g^{-1}(\der q,\zeta)=Xq$ for any scalar function $q$ so by fundamental theorem of calculus $I(\der q)=0$. By construction we have $\partial_x q=f_0$. Modify $q$ with a cut-off function $\chi(x)$ so that $\chi(x)=0$ when $x>1$ and $\chi(x)=1$ when $x\le 1/2$ and denote this modified function $q \chi$ also by $q$. Setting $\Tilde{f}:=f-\der q$ and using everything we have shown above proves the claim.
\end{proof}

From here on we will identify any smooth 1-form $f$ with $I f=0$ to the gauge corrected 1-form $\Tilde{f}=f-\der q$ of lemma~\ref{lma:gauge correction gas giant}. Next we will show a high-order boundary determination result for which the proof follows the same ideas as~\cite[proof of theorem 2.1]{LSU2003}.

\begin{proposition}[Boundary determination]
\label{prop:boundary determination}
Let $(M,g)$ be a gas giant.
    Let $f$ be a smooth 1-form so that $I f=0$ and assume we have applied the gauge correction of lemma~\ref{lma:gauge correction gas giant} to $f$. Then in boundary normal coordinates $(x,y)\in [0,1)\times \partial M$ we have
    \begin{equation}
    \frac{\partial^k f_i}{\partial x^k}(0,y)=0 \teksti{for} 0\le k\in \N.
    \end{equation}
\end{proposition}

\begin{proof}
    Fix the point $y\in \partial M$. We begin by proving the case $k=0$. Suppose to the contrary that $f_i(0,y)\not=0$ for $i=1,\dots,n-1$. Then without loss of generality we can assume that there is $\eta\in T^*_y \partial M$ with $\abs{\eta}_h=1$ so that 
    \begin{equation}
        h^{ij}(0,y) f_i(0,y) \eta_j>0.
    \end{equation}
    By continuity there is a neighborhood $U\subset T^* \partial M$ of the point $(y,\eta)$ where we also have
    \begin{align}
    \label{eqn:positivity in boundary determination}
        h^{ij}(0,y) f_i(0,y) \eta_j>0\teksti{for} (y,\eta)\in U.
    \end{align}
    Let $\gamma_s$ be the short geodesic with the initial condition $(0,y,\sigma_s)$ as in proposition~\ref{prop:short_geodesics}. Then for small $\varepsilon>0$ and for any $s\in (0,\varepsilon )$ the geodesic $\gamma_s$ has non-zero length and $\gamma_s$ is maximal. Hence $(I f)(0,y,\sigma_s)=0$. By the gauge correction of lemma~\ref{lma:gauge correction gas giant} we have
    \begin{equation}
    \begin{split}
        I(f)(0,y,\sigma_s)
        &=
        \int_0^{\tau(0,y,\sigma_s)} g^{\mu\nu}(z(t))f_\mu (z(t))\zeta_\mu (t)\der t
        \\
        &=
        \int_0^{\tau(0,y,\sigma_s)}x(t)^2 h^{ij}(z(t))f_i(z(t))\eta_j(t)\der t.
    \end{split}
    \end{equation}
    We Taylor expand $f$ in $x$ at $x=0$ and then do the same for the components of the metric $h$ to get
    \begin{equation}
    \begin{split}
    \label{eqn:Taylored integral}
        I(f)(0,y,\sigma_s)
        &=
        \int_0^{\tau(0,y,\sigma_s)} \left [x(t)^{2}h^{ij}(z(t)) f_i(0,y(t))\eta_j(t)+x(t)^{3} F(x(t),y(t),\eta (t)) \right]\der t
        \\
        &=
        \int_0^{\tau(0,y,\sigma_s)} \left [x(t)^{2}h^{ij}(0,y(t)) f_i(0,y(t))\eta_j(t)+x(t)^{3}P(x(t),y(t),\eta(t)) \right]\der t
    \end{split}
    \end{equation}
    By smoothness of $f$ and $h$ the functions $F$ and $P$ are uniformly bounded in $x, y$ and $\eta$. Taking $s$ small enough ensures that $(y(t),\eta(t))\in U$ by proposition~\ref{prop:short_geodesics} and by lemma~\ref{lma:height estimate} also the values $x(t)$ will stay small. Hence the value of the integral will be determined by the leading term of order $x(t)^2$ and for small enough $s$ we have by the estimate~\eqref{eqn:positivity in boundary determination} that
    \begin{equation}
        \int_0^{\tau(0,y,\sigma_s)} x(t)^{2}h^{ij}(0,y(t)) f_i(0,y(t))\eta_j(t)\der t>0
    \end{equation}
    from which we obtain $(I f)(0,y,\sigma_s)>0$ which is a contradiction.

    For $k\ge 1$ we prove the claim by induction via a contradiction in a similar way as above. Suppose $\partial^k_x f_i(0,y)=0$ holds for all $0\le k<l$ but $\partial^l_x f_i(0,y)\not=0$ for $i=1,\dots ,n-1$. Without loss of generality we can assume that there is $\eta\in T^* _y\partial M$ with $\abs{\eta}_h=1$ so that
    \begin{align}
    \label{eqn:contrary assumption}
        h^{ij}(0,y)\partial_x^l f_i(0,y) \eta_j>0.
    \end{align}
    We then proceed as in the case $k=0$ to show that $(If)(0,y,\sigma_s)>0$. The only difference is that in the step~\eqref{eqn:Taylored integral} where we take the Taylor expansion of $f$ the leading order term is of order $x(t)^{2+l}$ and contains the term~\eqref{eqn:contrary assumption} modulo a positive constant $1/l!$. Hence taking $s$ small enough we still arrive at the contradiction $(I f)(0,y,\sigma_s)>0$ which proves the claim by induction principle.
\end{proof}

We collect everything in this section to prove lemma~\ref{keylemma:bdy-det}

\begin{proof}[Proof of lemma~\ref{keylemma:bdy-det}]
    Let $f$ be a smooth 1-form with $If=0$ on a gas giant $(M,g)$. By lemma~\ref{lma:gauge correction gas giant} and proposition~\ref{prop:boundary determination} there exists $q\in C^\infty (M)$, supported near $\partial M$ with $q|_{\partial M} =0$, so that $f-\der q$ vanishes up to the boundary at polynomial order $k$ for any integer $k>0$. Hence by definition $f-\der q\in x^\infty C^\infty (M,T^* M)$.
\end{proof}

\section{Regularity of the integral function}
\label{sec:Regularity of the integral function}

In this section we prove the lemma~\ref{keylemma:u-regularity}. Throughout the proof we assume that $If=0$ and that we have done the boundary determination and gauge correction so that $f\in x^\infty  C^\infty(M,T^* M)$. To show that $u^f\in \Omega$ we prove the following weaker but more precise claim.

\begin{proposition}
\label{prop:blow-up rates of u}
Let $(M,g)$ be a non-trapping gas giant.
Let $k\in \N$, $f\in x^k C^\infty(M,T^* M)$ and $If=0$. Then $u^f\in x^{k+1} C^\infty (S^* M^\circ)$, $Xu^f=-\lambda f\in  x^k C^\infty (S^* M)$ and
\begin{align}
    \abs{\gradv u^f}_g^2&=\Order(x^{2k-2}),\quad &\abs{\gradh u^f}_g^2&=\Order(x^{2k-2}),\\
    \abs{\gradv X u^f}^2_g&=\Order(x^{2k-4}),\quad &\abs{X\gradv u^f}^2_g&=\Order(x^{2k-4}) 
\end{align}
in the interior $S^* M^\circ$ and $\gradv u^f|_{\partial S^* M}=0$.
\end{proposition}

We will first show that lemma~\ref{keylemma:u-regularity} follows directly from proposition~\ref{prop:blow-up rates of u}.

\begin{proof}[Proof of lemma~\ref{keylemma:u-regularity}]
    Since the claim of proposition~\ref{prop:blow-up rates of u} holds for all $k$ then we may assume $If=0$ and $f\in x^\infty C^\infty (M,T^* M)$ on a non-trapping gas giant $(M,g)$ and adapt the results of~\ref{prop:blow-up rates of u} to this setting. Clearly also $u^f\in x^\infty C^\infty(S^* M^\circ)$ and $Xu^f\in C^\infty (S^* M)$. The rest of the quantities containing derivatives in proposition~\ref{prop:blow-up rates of u} are used to compute the corresponding $L^2$ norms of $\gradv u^f,\gradh u^f,\gradv Xu^f$ and $X\gradv u^f$ respectively. Since the estimates hold for all $k$ then we also have $\nabla_{S^* M} u^f\in x^\infty L^2(S^* M,TS^* M)$ and $\gradv Xu^f ,X\gradv u^f\in L^2(N)$. Since the gas giant metric $g$ degenerates these estimates up to the boundary then clearly also similar estimates hold without the $g$-norm. Hence also $\nabla_{S^* M} u^f\in x^\infty L^\infty (S^* M,TS^* M)$. Clearly by definition of the integral function and by proposition~\ref{prop:blow-up rates of u} $u^f|_{\partial S^* M}=\gradv u^f|_{\partial S^* M}=0$. Hence $u^f\in \Omega$ and the proof of lemma~\ref{keylemma:u-regularity} is concluded.
\end{proof}

We continue by proving various results that we will then compile at the end of the section to prove proposition~\ref{prop:blow-up rates of u}. We begin with analyzing the smoothness of the integrand $\lambda f$ of $u^f$ and the smoothness of the exit time of geodesics $\tau(z,\zeta)$. We have the following lemma.

\begin{lemma}
\label{lma:regularity of function}
Let $(M,g)$ be a non-trapping gas giant.
    Let $k\in \N$, $f\in x^k C^\infty(M,T^* M)$ and $If=0$. We identify the 1-form $f$ to a function on $S^* M$ by
    \begin{equation}
    \lambda f(z,\zeta)=g^{\mu\nu }(z)\zeta_\mu f_\nu(z) \teksti{for all} (z,\zeta)\in S^* M.
    \end{equation}
    Then $\lambda f\in C^\infty(S^* M)$ and for all $(z,\zeta)\in S^* M$ near the boundary we have $|\lambda f(z,\zeta)|\le Cx^k$ and $|\der\,  \lambda f(z,\zeta)|_e\le Cx^{k-1}$. 
\end{lemma}

\begin{proof}
    It is clear from the definition of $\lambda f$ that the function is smooth in the interior. Hence it remains to show that the function is smooth in any neighborhood that contains the structure of boundary normal coordinates. In the coordinates $(x,y,\xi,\eta)$ we have
    \begin{equation}
        \lambda f(z,\zeta)=\lambda f(x,y,\xi,\eta)=f_0(x,y)\xi+x^2 h^{ij}(x,y)f_i(x,y)\eta_j
    \end{equation}
By definition the expression is smooth w.r.t.\ all the variables $x,y,\xi$ and $\eta$. The leading order vanishing rate to the boundary is given by the term $|f_0\xi|\le Cx^k$ where $C$ is independent of $x,y,\xi$ and $\eta$. To prove the claim $|\der\,  \lambda f(z,\zeta)|_e\le Cx^{k-1}$ we compute the partial derivatives of $\lambda f$:
\begin{align}
    \partial_x\lambda f(z,\zeta)&=\partial_x f_0\, \xi+2xh^{ij}f_i\eta_j+x^2\partial_x h^{ij} f_i\eta_j+x^2h^{ij}\partial_x f_i\, \eta_j,\\
    \partial_{y^k}\lambda f(z,\zeta)&=\partial_{y^k} f_0\, \xi +x^2\partial_{y^k}h^{ij} f_i\eta_j+x^2h^{ij}\partial_{y^k}f_i\, \eta_j,\\
    \partial_\xi \lambda f(z,\zeta)&=f_0,\\
    \partial_{\eta_k}\lambda f(z,\zeta)&=x^2h^{ik}f_i
\end{align}
from which we see that if we write $f_0=x^k\bar{f}_0$ for $k\ge 1$, where $\bar{f}_0\in C^\infty (M)$ is independent of $x$, then the leading order behavior is given by $\partial_x f_0 \xi =kx^{k-1} \bar{f}_0\xi +\Order(x^k)$.
\end{proof}

\begin{corollary}
\label{cor:regularity of lambda f}Let $(M,g)$ be a non-trapping gas giant.
    Let $\lambda f$ be as in lemma~\ref{lma:regularity of function}. Then $X u^f\in x^kC^\infty (S^* M)$ and
    \begin{equation}
    \abs{\gradv X u^f}^2_g=\abs{\gradv \lambda f}_g^2=\Order(x^{2k-4}).
    \end{equation}
\end{corollary}

\begin{proof}
    Since the geodesic flow on the cotangent bundle is smooth and $f$ vanishes up to the boundary then by fundamental theorem of calculus $Xu^f=-\lambda f$ and by lemma~\ref{lma:regularity of function} $-\lambda f\in x^kC^\infty (S^* M)$. To prove the second claim we restrict to a neighborhood containing boundary normal coordinates and use the local coordinate formula of the vertical gradient given by equation~\eqref{eqn:vertical gradient local} found in the appendix to get
    \begin{equation}
        \abs{\gradv \lambda f}_g^2=g_{\mu\nu}\partial_{\zeta_\mu}\lambda f\, \partial_{\zeta_\nu} \lambda f\le C x^{-2}\abs{\der\, \lambda f}_e^2\le Cx^{2(k-1)-2}=Cx^{2k-4}.
    \end{equation}
\end{proof}

In the following proofs we need some estimates for the normal Jacobi fields for which we will state the following lemma.

\begin{lemma}[{\cite[Lemma 26]{dHIKM2024}}]
    \label{lma:Jacobi fields}
    Let $(M,g)$ be a gas giant.
    Let $J(t)$ be a Jacobi field everywhere normal to a bicharacteristic curve $(z(t),\zeta(t))$ with $x(0)<\varepsilon$ and $\xi(0)\le 0$. Then $\abs{J(t)}_{\overline{g}}\le C$ and $\abs{\Der_t J(t)}_{\overline{g}}\le C x(t)^{-1}$ for all $t\in [0,\tau (z(0),\zeta(0))]$.
\end{lemma}

The proof of the lemma can be found in~\cite{dHIKM2024} in a different smooth structure but it is easy to check that the proof and hence also the result will be the same in the boundary normal coordinate smooth structure that we use. Next we prove the following proposition for the exit time of the geodesic flow:

\begin{proposition}
    \label{prop:Sasaki gradient of exit time}
    Let $(M,g)$ be a non-trapping gas giant.
    The Sasaki gradient of the exit time is given by $\nabla_{S^* M}\tau =-X$. In particular $\tau\in C^\infty(S^* M)$.
\end{proposition}

\begin{proof}
    We begin by proving the claim in the interior points $S^* M^\circ$ and then using a smooth extension of the manifold $M$ w.r.t.\ symplectic structure to prove the claim in $\partial T^* M\setminus \partial_0 T^* M$ and hence also in $\partial S^* M\subset \partial T^* M\setminus \partial_0 T^* M$ (inclusion follows from the gnomonic projection discussion in section~\ref{sec:Boundary determination}). 
    
    First we show that $X\tau(z,\zeta)=-1$ for all $(z,\zeta)\in S^* M^\circ$. Since the geodesic flow on the cotangent bundle is smooth then
    \begin{align}
    \label{eqn:X derivative of exit time}
        X\tau(z,\zeta)=\frac{\der }{\der s}\tau(\phi_s(z,\zeta))|_{s=0}=\frac{\der }{\der s}(\tau(z,\zeta)-s)|_{s=0}=-1.
    \end{align}

    Next we show that $\gradv\tau (z,\zeta) =\gradh \tau (z,\zeta)=0$ for all $(z,\zeta)\in S^* M^\circ$. To this end let $s\in (-1,1)$ and $c(s)$ be a curve in $S^* M^\circ$ with $c(0)=(z,\zeta)$, $\dot{c}(0)=\theta\in T_{z,\zeta} S^* M$ and $\theta\perp X$. We use the implicit function theorem on the function $F\colon \R \times \R \to \R $, $F(s,t)=x(\pi(\phi_t(c(s)))$ at the point $(s,\tau(c(s)))$ and value $s=0$. Here we assume that the scaling $s$ in $c(s)$ is chosen so that $\phi_{\tau(c(s))}(c(s))$ is within a single patch of boundary normal coordinates for all $s\in (-1,1)$ so that $F$ is well-defined. Observe that $F(0,\tau(z,\zeta))=0$ and
    \begin{equation}
        \begin{split}
        \partial_t F(0,t)|_{t=\tau (c(0))}&=\partial_t F(0,t) |_{t=\tau (z,\zeta)}\\
        &=\der x(\der \pi (\dot{\phi}_{\tau(z,\zeta)}(c(0))))\\
        &=\der x(\der \pi (\dot{\phi}_{\tau(z,\zeta)}(z,\zeta)))
        \end{split}
    \end{equation}
    where $\dot{\phi}_{\tau(z,\zeta)}(z,\zeta)=(\dot{\gamma}_{f},\dot{\zeta}_{f})$. Hence
    \begin{equation}
        \partial_t F(0,t)|_{t=\tau(z,\zeta)}=\der x( \der \pi (\dot{\gamma}_{f},\dot{\zeta}_{f}))=\der x (\dot{\gamma}_{f})=\dot{\gamma}_{f}^0.
    \end{equation}
    By Hamilton's equations and the unit speed constraint of geodesics we know that $\dot{\gamma}_{f}^0= -1$ so $\partial_t F(0,t)|_{t=\tau(z,\zeta)}$ is non-zero and hence we may apply the implicit function theorem. Then
    \begin{equation}
        \frac{\der F(s,\tau(c(s)))}{\der s}=\partial_1 F(s,\tau (c(s)))+\partial_2 F(s,\tau (c(s))) \frac{\der \tau (c(s))}{\der s}=0
    \end{equation}
    so that
    \begin{equation}
    \frac{\der \tau (c(s))}{\der s}|_{s=0}=-\frac{\partial_1 F(s,\tau(c(s)))}{\partial_2 F(s,\tau(c(s)))}\Bigg|_{s=0}.
    \end{equation}
    Next observe that $\pi_{T^* M}(\phi_t(z,\zeta))=\pi_{TM}(\gamma(t),\dot{\gamma}(t))=\gamma(t)$ so $\der \pi (\der \phi_t(\theta))=J_\theta(t)$, where, by $\theta\perp X$, $J_\theta$ is the normal Jacobi field along the geodesic $\gamma$ with the initial condition $J_\theta(0)=\der \pi(\theta)$. Then by lemma~\ref{lma:Jacobi fields} we know that the values of the components of $J_\theta$ are bounded so that
    \begin{equation}
        \partial_s F(s,t)|_{s=0}=\der x(\der \pi(\der \phi_t(\partial_s c(s)))) |_{s=0}=\der x (J_\theta(t))
    \end{equation}
    and
    \begin{equation}
        \partial_s F(s,\tau(z,\zeta))|_{s=0}=\der x (J_\theta(\tau(z,\zeta)))=J_\theta^0(\tau(z,\zeta))=0
    \end{equation}
    where the zero value follows from the fact that $J_\theta$ is a normal Jacobi field so in particular it has to be normal to $\dot{\gamma}_f=(-1,0,\dots ,0)$ at $\tau(z,\zeta)$.
    Hence we get 
    \begin{equation}
    \frac{\der \tau (c(s))}{\der s}|_{s=0}=0\teksti{in} S^*M^\circ
    \end{equation}
    which proves $\gradh\tau(z,\zeta)=\gradv\tau(z,\zeta)=0$ for all $(z,\zeta)\in S^* M^\circ$

    To prove the claim in $S^* M$ we will prove it on $\partial T^* M\setminus \partial_0 T^* M$ via the same implicit function theorem proof as above and note that by the boundary determination we may identify $\partial T^* M\setminus \partial_0 T^* M$ to $\partial S^* M$. Since we are taking variations on the boundary $\partial M$ it will be useful to extend the manifold. First we consider the Riemannian manifold $(M,\overline{g})$ with a smooth metric $\overline{g}$. By~\cite[lemma 3.1.8]{PSU2023} there is a closed smooth extension $(N,\overline{g})$ of $(M,\overline{g})$, see also~\cite[Example 9.32]{Lee2013}. Since $(N,\overline{g})$ is not expected to carry any useful information about the singular gas giant metric $g$, we drop the Riemannian structure from here on and consider the smooth manifold $M$ and its closed smooth extension $N$ (not to be confused with the notation of the normal bundle). We will extend various smooth functions from $M$ to $N$ by the extension theorem of Whitney~\cite{Whi1936}.
    
    We start by proving that there exists a smooth extension of the geodesic flow of a gas giant from $M$ to $N$. Since $N$ is a smooth extension of $M$ then the cotangent bundle $T^* N$ is a smooth extension of $T^* M$. Hence also the canonical symplectic form $\omega =\sum_{\mu =0}^{n-1}\der z^\mu \wedge \der \zeta_\mu $ extends smoothly from $T^* M$ to $T^* N$. By Whitney extension theorem we may extend the smooth Hamiltonian $H\colon T^*M\to \R $, $H(z,\zeta)=\frac{1}{2}g^{\mu \nu }(z)\zeta_\mu\zeta_\nu $ as a smooth function in $T^* N$. On the level of symplectic geometry the corresponding Hamiltonian flow is given by the equation $\der H=i_X\omega$ and since both $H$ and $\omega$ have smooth extensions in $T^* N$ then so does the Hamiltonian vector field of geodesic flow $X$. Hence also the geodesic flow has a smooth extension $\phi_t\colon T^* N\to T^* N$. With this we can compute the Sasaki gradient of the exit time on $\partial T^* M\setminus \partial_0 T^* M$.

    Let $s\in (-1,1)$ and $c(s)$ be a curve in $T^* N$ with $c(0)=(z,\zeta)\in \partial T^* M\setminus \partial_0 T^* M$ and $\dot{c}(0)=\theta\in T_{z,\zeta}(\partial T^* M\setminus \partial_0 T^* M)$ with $\theta\perp X$. Let $F(s,t)=x(\pi(\phi_t(c(s)))$ be the smoothly extended version of the function used above. Again we use the implicit function theorem at a point $F(0,\tau(z,\zeta))=0$. The main observation is that since $(z,\zeta)\in \partial T^* M\setminus \partial_0 T^* M$ then by Hamilton's equations and boundary determination we know that $\dot{\gamma}_f^0\not=0$. Hence the same calculation as in the case of $S^* M^\circ$ gives us
    \begin{equation}
    \frac{\der \tau (c(s))}{\der s}\Big|_{s=0}=0\teksti{on} \partial T^* M\setminus \partial_0 T^* M
    \end{equation}
    which in particular proves $\gradh \tau(z,\zeta)=\gradv \tau (z,\zeta)=0$ for all $(z,\zeta)\in \partial S^* M\subset \partial T^* M\setminus \partial_0 T^* M$. 

    Next we will prove that $X\tau (z,\zeta )=-1$ on $\partial T^* M\setminus \partial_0 T^* M$. Note that the calculation in equation~\eqref{eqn:X derivative of exit time} does not work on the boundary so we will instead use the implicit function theorem. The main difference with the previous calculation is that we take $\dot{c}(0)=\theta\in T_{z,\zeta}(\partial T^* M\setminus \partial_0 T^* M )$ with $\theta \parallel X$. Let $A>0$ be some real number. Then locally we can write $(z,\zeta)=(z,\pm A,\eta)$ so that $X=\pm A\partial_x$ at $(z,\zeta)\in \partial T^* M\setminus \partial_0 T^* M$. Since we are only interested in taking a variation w.r.t.\ $X$ then by $X\parallel \theta$ we restrict to the case where $\der \pi (\der \phi_t (\theta))=J_\theta (t)$ and $J_\theta$ is the parallel Jacobi field characterized by $J_\theta(0)=(\pm A,0,\dots ,0)$. By the constant speed condition of the flow we have $\dot{\gamma}_f=(- A,0,\dots ,0)$ and $J_\theta (\tau(z,\zeta))=\dot{\gamma}_f$ so by applying the implicit function theorem to $F$ as above gives us
    \begin{equation}
        X\tau (z,\zeta)=-\frac{J^0_\theta (\tau(z,\zeta))}{\dot{\gamma}_f^0}=-\frac{- A}{- A}=-1\teksti{on} \partial T^* M\setminus \partial_0 T^* M.
    \end{equation}
    Restricting to $\partial S^* M\subset \partial T^* M\setminus \partial_0 T^* M$ finishes the proof of the Sasaki gradient of the exit time. Since $\nabla_{S^* M}\tau=-X$ then all further derivatives will be zero so it follows that $\tau\in C^ \infty (S^* M)$.
\end{proof}

Proposition~\ref{prop:Sasaki gradient of exit time} says that the exit time on a gas giant is geodesic. This is unique to gas giant geometry and is false on standard Riemannian manifolds, which is quickly verified, for example, on the unit disk in $\R^2$ equipped with the Euclidean metric.

Next we prove the first statement of proposition~\ref{prop:blow-up rates of u}:

\begin{lemma}
    \label{lma:asymptotics of integral function}
    Let $(M,g)$ be a non-trapping gas giant.
    Let $k\in \N$, $f\in x^k C^\infty(M,T^* M)$ and $If=0$. Then $u^f\in x^{k+1} C^\infty (S^* M^\circ)$.
\end{lemma}

\begin{proof}

The fact that $u^f\in C^\infty (S^* M^\circ)$ follows directly from $\lambda f\in C^\infty (S^* M)$ and $\tau\in C^\infty (S^* M)$, cf.\  lemma~\ref{lma:regularity of function} and proposition~\ref{prop:Sasaki gradient of exit time}. To prove that $u^f(z,\zeta)=\Order (x^{k+1})$ where the constants are uniform in $(x,y,\xi,\eta)$ let $z\in M^\circ$. By assumption $u^f (z,\zeta)+u^f(z,-\zeta)=I f= 0$. Namely we have
    \begin{equation}
    \abs{u^f(z,\zeta)}=\abs{u^f(z,\zeta)+u^f(z,-\zeta)-u^f(z,-\zeta)}=\abs{u^f(z,-\zeta)}
    \end{equation}
    so the order of magnitude depends only on the initial base point $z$ and we can choose the direction so that the covector $\zeta$ locally points to the boundary, i.e. $\zeta=(\xi,\eta)$ with $\xi\le 0$. Suppose we take $z$ close to the boundary and $\zeta$ so that $x$ is strictly decreasing along the geodesic $\gamma_{z,\zeta}$. Denote any uniform constant by $C>0$. Then
\begin{equation}
\begin{split}
    \abs{u^f(z,\zeta)}
    &\le
    \int_0^{\tau(z,\zeta)}\abs{f_\mu (\gamma_{z,\zeta}(t))\dot{\gamma}^\mu_{z,\zeta}(t)}\der t
    \\
    &\le
    C\int _0^{\tau(z,\zeta)}x(t)^k \der t
    \\
    &\le
    C\int_0^{\tau(z,\zeta)}x(0)^k\der t
    \\
    &=C\tau(z,\zeta)x(0)^k
    \\
    &=Cx(0)^{k+1}+\Order\left ( x(0)^{k+3} \right)
\end{split}
\end{equation}
    since $\tau(x,y,\xi,\eta)=x+\Order(x^{3})$ whenever $\xi\le 0$ by proposition~\ref{prop:asymptotic of exit time}.
\end{proof}

To conclude the section we prove two lemmas regarding the derivatives of the integral function.

\begin{lemma}
\label{lma:asymptotics of derivatives of integral funciton}
    Let $(M,g)$ be a non-trapping gas giant.
    Let $k\in \N$, $f\in x^k C^\infty(M,T^* M)$ and $If=0$. Then in the interior $S^* M^\circ$ we have
    \begin{equation}
        \abs{\gradv u^f}_g^2=\Order(x^{2k-2}),\quad \abs{\gradh u^f}_g^2=\Order(x^{2k-2}),\quad 
        \abs{X\gradv u^f}^2_g=\Order(x^{2k-4}).
    \end{equation}
\end{lemma}

\begin{proof}
    We begin by computing the vertical and horizontal gradients of $u^f$. To clear up notation we omit the identification $\lambda$ from the following calculations and let the arguments define if $f$ is a function or a 1-form. Let $c(s)$ be a curve in $S^* M^\circ$ so that $c(0)=(z,\zeta)$ and $\dot{c}(0)=\theta$ for some $\theta\in T_{z,\zeta} S^* M^\circ$, with $\theta \perp X$. Then by proposition~\ref{prop:Sasaki gradient of exit time} and smoothness of $f$
    \begin{equation}
    \begin{split}
        \partial_\theta u^f(z,\zeta)
        &=
        \frac{\der }{\der s}\int_0^{\tau(c(s))}f(\phi_t(c(s)))\der t |_{s=0}\\
        &=f(\phi_{\tau(z,\zeta)}(z,\zeta))\frac{\der }{\der s}\tau(c(s))|_{s=0} +\int_0^{\tau(z,\zeta)}\frac{\der }{\der s}f(\phi_t(c(s)))|_{s=0}\der t\\
        &=\int_0^{\tau(z,\zeta)}\frac{\der }{\der s}f(\phi_t(c(s)))|_{s=0}\der t
    \end{split}
    \end{equation}

    Let $J_\theta(t)$ be a normal jacobi field along $\gamma_t=\pi (\phi_t(z,\zeta))$ with $J_\theta(0)=\der \pi (\theta)\in \Hor$ and $\Der_t J_\theta(0)=K(\theta)\in \Ver$. Since the initial condition of the geodesic flow is in the interior $S^* M^\circ$ then we can identify the geodesic flow from the cotangent bundle to the tangent bundle via the gas giant metric $g$ and hence 
    \begin{equation}
    \der \phi_t(\theta)=(\der \pi (\der \phi_t(\theta)),K(\der \phi_t(\theta)))=(J_\theta(t),\Der_t J_\theta (t))\in \Hor\oplus \Ver.
    \end{equation}
    By chain rule we have that 
    \begin{equation}
    \begin{split}
    \frac{\der}{\der s} f(\phi_t(c(s)))\Big|_{s=0}
    &=
    \der f\frac{\partial }{\partial s}\phi_t(c(s))\Big|_{s=0}
    \\
    &=
    \der f\left(\der \phi_t \frac{\partial }{\partial s}c(s)\Big|_{s=0}\right)
    \\
    &=
    \der f(\der \phi_t(\theta))
    \\
    &=
    \der f(J_\theta(t),\Der_t J_\theta(t)).
    \end{split}
    \end{equation}
    so that combining everything gives us
    \begin{equation}
    \gradv u^f(z,\zeta)=\int_0^{\tau(z,\zeta)} f(\Der_t J_\theta(t))\,\der t
    \end{equation}
    and
    \begin{equation}
    \gradh u^f(z,\zeta)=\int_0^{\tau(z,\zeta)}\der f(J_\theta(t),\dot{\gamma}_{z,\zeta} (t))\,\der t.
    \end{equation}
    Using the bounds of lemma~\ref{lma:Jacobi fields} for Jacobi fields, lemma~\ref{lma:regularity of function} for the blow-up estimates of $f$ and proposition~\ref{prop:asymptotic of exit time} for the asymptotic of the exit time gives us in a similar fashion as in the proof of lemma~\ref{lma:asymptotics of integral function} that
\begin{equation}
\begin{split}
        \abs{\gradv u^f(z,\zeta)}_g^2
        &\le
        C x^{-2}\abs{\gradv u^f(z,\zeta)}_e^2
        \\
        &\le
        Cx^{-2}
        \left ( \int_0^{\tau(z,\zeta)}
        \abs{f(\Der_t J_\theta(t))}
        \der t
        \right )^2
        \\
        &\le
        C x^{-2}(\tau(z,\zeta) x^{k-1})^2
        \\
        &\le
        Cx^{2k-2}
\end{split}
\end{equation}
and
\begin{equation}
\begin{split}
        \abs{\gradh u^f(z,\zeta)}_g^2
        &\le
        C x^{-2}\abs{\gradh u^f(z,\zeta)}_e^2
        \\
        &\le
        Cx^{-2}\left (  \int_0^{\tau(z,\zeta)} \abs{\der f(J_\theta(t),\dot{\gamma}_{z,\zeta} (t))}\,\der t \right )^2
        \\
        &\le
        C x^{2k-2}
\end{split}
\end{equation}
    where $C$ denotes any uniform constant. To prove the last claim we use the commutator formula $[X,\gradv ]=-\gradh$ in $S^* M^\circ$ and corollary~\ref{cor:regularity of lambda f} to compute
    \begin{equation}
    \begin{split}
        \abs{X\gradv u^f(z,\zeta)}_g^2
        &=
        \abs{\gradv Xu^f(z,\zeta)-\gradh u^f(z,\zeta)}^2_g
        \\
        &\le
        \abs{\gradv Xu^f(z,\zeta)}_g^2+\abs{\gradh u^f(z,\zeta)}_g^2
        \\
        &\le
        C\left(x^{2k-4}+x^{2k-2}\right)
        \\
        &\le
        2Cx^{2k-4}
    \end{split}
    \end{equation}
    where $C$ is a uniform constant and we used the arbitrary estimate $x^2\le 1$.
\end{proof}

\begin{lemma}
\label{lma:vertical gradient on boundary}
    Let $(M,g)$ be a non-trapping gas giant and let~$k\in \N$, $f\in x^k C^\infty(M,T^* M)$. If $If = 0$ then $\gradv u^f|_{\partial S^* M}=0$.
\end{lemma}

\begin{proof}
    Fix $y\in \partial M$. Let $c:(-1,1)\to \partial S^* M$ be a smooth path with
    \begin{equation}
    c(s)
    =
    (0,y,\zeta(s))
    =
    (0,y,\pm 1, \eta (s)).
    \end{equation}
    For $(0,y,1,\eta(s))$ we have $u^f(c(s))=If= 0$ since, e.g.\  by the gnomonic projection discussion in section~\ref{sec:Boundary determination}, $(0,y,1,\eta(s))$ is an initial condition of a maximal geodesic for all $s\in (-1,1)$. For $(0,y,-1,\eta(s))$ the corresponding maximal geodesic will be stationary so also in this case $u^f(c(s))=0$ for all $s\in (-1,1)$. Hence
    \begin{equation}
    \gradv u^f|_{\partial S^* M}=\frac{\der }{\der s}u^f(c(s))\Big|_{s=0}=\frac{\der }{\der s} 0\Big|_{s=0}=0
    \end{equation}
    which proves the claim.
\end{proof}

We combine everything proven in this section to prove proposition~\ref{prop:blow-up rates of u}.

\begin{proof}[Proof of proposition~\ref{prop:blow-up rates of u}]
    Let $k\in \N$, $f\in x^k C^\infty (M,T^* M)$ and $If=0$ on a non-trapping gas giant $(M,g)$. By lemma~\ref{lma:asymptotics of integral function} we have $u^f\in x^{k+1} C^\infty (S^* M^\circ)$ and by corollary~\ref{cor:regularity of lambda f} we have $Xu^f\in C^\infty (S^* M)$. By lemma~\ref{lma:asymptotics of derivatives of integral funciton} and corollary~\ref{cor:regularity of lambda f} we have the blow-up rates of the derivatives in the claim of proposition~\ref{prop:blow-up rates of u}. Finally by lemma~\ref{lma:vertical gradient on boundary} we have $\gradv u^f |_{\partial S^* M}=0$ and this concludes the proof of proposition~\ref{prop:blow-up rates of u}.
\end{proof}

\section{Pestov identity}
\label{sec:Pestov identity}

In this section we prove a Pestov identity for functions $u\in \Omega$ on a gas giant and use the Pestov identity to prove the lemma~\ref{keylemma:pestov}.

\begin{proposition}
\label{prop:Pestov_with_regularities}
    Let $(M,g)$ be a non-trapping gas giant with $R\le 0$ and $u\in \Omega$ in the sense of definition~\ref{def:niceness-of-uf}. Then we have
    \begin{equation}
    \norm{\gradv Xu}^2\ge \norm{X\gradv u}^2+(n-1)\norm{Xu}^2.
    \end{equation}
\end{proposition}

\begin{proof}
    Let $M_\varepsilon=\{x\ge \varepsilon\}$ denote the truncation manifold for $\varepsilon>0$. Then $u\in C^\infty(S^* M_\varepsilon)$ so by~\cite{GGSU2019} the following Pestov identity holds
    \begin{align}
    \label{eqn:Pestov truncated}
        \norm{\gradv X u}_\varepsilon^2=\norm{X\gradv u}_\varepsilon^2-(R\gradv u,\gradv u)_\varepsilon+(n-1)\norm{Xu}_\varepsilon^2+B_\varepsilon(u)
    \end{align}
    where $B_\varepsilon(u)$ is a boundary term on $\partial S^* M_\varepsilon$ defined as
    \begin{equation}
        \int_{\partial S^* M_\varepsilon} \ip{\gradv u}{\gradh u}-(n-1) uXu \, \der \sigma
    \end{equation}
    and the subscript $\varepsilon$ denotes restriction of the norms to $L^2 (S^* M_\varepsilon)$. We will show that this identity converges as $\varepsilon \to 0$ in the weaker setting where $u$ has all the same properties as $u^f$ in  proposition~\ref{prop:blow-up rates of u}, namely we assume that $u\in x^{k+1} C^\infty (S^* M^\circ )$, $Xu\in x^k C^\infty (S^* M)$ and we have the asymptotic rates 
    \begin{align}
    \abs{\gradv u}_g^2&=\Order(x^{2k-2}),\quad &\abs{\gradh u}_g^2&=\Order(x^{2k-2}),\\
    \abs{\gradv X u}^2_g&=\Order(x^{2k-4}),\quad &\abs{X\gradv u}^2_g&=\Order(x^{2k-4}) 
    \end{align}
    for any integer $k$. In the following calculations we take the integer $k$ large enough, $k\ge n+1$ to be precise, so that various upper bounds converge to zero as $\varepsilon\to 0$.

    First we have
\begin{equation}
\begin{split}
    \abs{B_\varepsilon(u)}
    &\le \int_{\partial S^* M_\varepsilon} \abs{\ip{\gradv u}{\gradh u}}+\abs{(n-1)uXu}\, \der \sigma
    \\
    &\le
    \Vol^{n-1} (x=\varepsilon )\omega_n\left [\abs{\gradv u}_{g,x=\varepsilon}\abs{\gradh u}_{g,x=\varepsilon} +C(n-1)\varepsilon^{k+1+k} \right]
    \\
    &\le
    C'\varepsilon^{-(n-1)}\left [ \varepsilon^{2k-2}+(n-1)\varepsilon^{2k+1}\right],
\end{split}
\end{equation}
    where $\omega_n$ denotes the constant spherical measure of $S^*_z M_\varepsilon$ on $z\in \{x=\varepsilon\}$, and $C$ and $C'$ are some uniform constants. By choosing $k$ large enough we see that $B_\varepsilon (u)\to 0$ as $\varepsilon \to 0$. Next
\begin{equation}
\begin{split}
        \norm{X\gradv u}^2-\norm{X\gradv u}^2_\varepsilon
        &=
        \int_{S^* M\setminus S^* M_\varepsilon} \abs{X\gradv u}_g^ 2\der \Sigma
        \\
        &\le
        C\int_{S^* M\setminus S^* M_\varepsilon} x^{2k-4} \der \Sigma
        \\
        &\le
        C\Vol_{\overline{g}}(M\setminus M_\varepsilon ) \int_0^\varepsilon x^{2k-4-(n-1)}\der x
        \\
        &\le
        C\Vol_{\overline{g}}(M\setminus M_\varepsilon ) \varepsilon^{2k-2-n}.
\end{split}
\end{equation}
    By choosing $k$ large enough we see that $\norm{X\gradv u}_\varepsilon^2\to \norm{X\gradv u}^2$ as $\varepsilon\to 0$. Likewise
    \begin{equation}
    \norm{Xu}^2-\norm{Xu}^2_\varepsilon\le C \varepsilon^{2k+2-n}\teksti{and} \norm{\gradv Xu}^2-\norm{\gradv Xu}^2_\varepsilon\le C\varepsilon^{2k-2-n}
    \end{equation}
    so also $\norm{Xu}_\varepsilon^2\to \norm{Xu}^2$ and $\norm{\gradv Xu}_\varepsilon^2\to \norm{\gradv Xu}^2$ as $\varepsilon \to 0$ when $k$ is large enough. For completeness we will also show that the curvature term converges even though we already have all the necessary bits to prove the claim. In the smooth structure $(x,y)$ near the boundary we have $R\sim x^{-2}$ in the sectional curvature sense, cf.~\cite[proposition 1]{dHIKM2024}. Then
\begin{equation}
\begin{split}
    (R\gradv u,\gradv u)-(R\gradv u,\gradv u)_\varepsilon
    &\le
    C\int_{S^* M\setminus S^* M_\varepsilon} x^{-2} \abs{\gradv u}^2 _g \der \Sigma
    \\
    &\le
    C\int_{S^* M\setminus S^* M_\varepsilon} x^{-2+2k-2}\der \Sigma
    \\
    &\le
    C\varepsilon^{2k-2-n}
\end{split}
\end{equation}
    which shows that $(R\gradv u,\gradv u)_\varepsilon\to (R\gradv u,\gradv u)$ as $\varepsilon\to 0$. Hence for a function $u$, that fulfills the same properties as $u^f$ in proposition~\ref{prop:blow-up rates of u} for $k$ large enough, we get the Pestov identity
    \begin{align}
    \label{eqn:Pestov full}
        \norm{\gradv X u}^2=\norm{X\gradv u}^2-(R\gradv u,\gradv u)+(n-1)\norm{Xu}^2.
    \end{align}
    in the limit $\varepsilon\to 0$ by~\eqref{eqn:Pestov truncated}. Therefore the identity~\eqref{eqn:Pestov full} also holds under the stronger assumption $u\in \Omega$. Since $R\le 0$ then this proves the claim.
\end{proof}

\begin{corollary}
\label{cor:corollary of Pestov}
Let $(M,g)$ be a non-trapping gas giant with non-positive curvature.
    Let $u\in \Omega$ and $Xu(z,\zeta)$ be linear in $\zeta$ for all $z\in M$.
    Then $X\gradv u=0$.
\end{corollary}

\begin{proof}
    Let $\varepsilon>0$ and $M_\varepsilon$ be the truncated manifold as in the proof of proposition~\ref{prop:Pestov_with_regularities}.
   If $w\in x^\infty C^\infty (S^* M)$ and $w(z,\zeta)$ is linear in $\zeta$ for all $z$, then
   on every interior fiber we have (cf.~\cite[proof of lemma 4]{Ilm2018})
    \begin{align}
    	\label{eqn:sphere integral 1}
    	\int_{S^* _z M_\varepsilon} \abs{\gradv w(z,\zeta)}^2\der \zeta
        =
        (n-1)\int_{S^* _z M_\varepsilon} \abs{w(z,\zeta)}^2\der \zeta
        .
    \end{align}
    Therefore
    \begin{equation}
    \norm{\gradv w}_\varepsilon^2=(n-1) \norm{ w}_\varepsilon^2.
    \end{equation}
    By our assumptions on $u$ we may use this on $w=Xu$ and take the limit $\eps\to0$, so we get
    \begin{equation}
    \norm{\gradv Xu}^2=(n-1) \norm{ Xu}^2
    .
    \end{equation}
    Combining this with proposition~\ref{prop:Pestov_with_regularities} gives $\norm{X\gradv u}^2=0$.
\end{proof}

We conclude by proving the following lemma which is a slightly more general version of lemma~\ref{keylemma:pestov}.

\begin{lemma}
\label{lma:Pestov conclusion}
Let $(M,g)$ be a non-trapping gas giant with non-positive curvature. Let $u\in \Omega$ and $Xu(z,\zeta)$ be a polynomial of order 1 w.r.t.\  $\zeta$ for all $(z,\zeta)\in S^* M$. Then $\gradv u\equiv 0$ and there exists a function $p\in C^\infty (M)$ so that $u=\pi^* p$.
\end{lemma}

\begin{proof}
    Let $(z,\zeta)\in  S^* M^\circ$. Since $u\in \Omega$ then $\gradv u|_{\partial S^* M}=0$ and by smoothness of the flow and corollary~\ref{cor:corollary of Pestov} we have
\begin{equation}
\begin{split}
    0
    &=
    \int_0^{\tau(z,\zeta) }X\gradv u(\phi_t(z,\zeta))\der t
    \\
    &=
    \gradv u(\phi_{\tau(z,\zeta)}(z,\zeta))-\gradv u(z,\zeta)
    \\
    &=
    0-\gradv u(z,\zeta).
\end{split}
\end{equation}
    Combining this with the boundary value $\gradv u|_{\partial S^* M}=0$ then shows that $\gradv u\equiv 0$.
    
    Next observe that since $\gradv u \equiv 0$ then $u$ is independent of $\zeta\in S^*_zM$ for each $z\in M$. Write $u(z,\zeta)=p(z)$ for a scalar function $p$. By the assumption $u\in \Omega$ we have that $Xu\in C^\infty (S^* M)$ so 
    \begin{equation}
    Xu=X\pi^* p=\lambda\der p
    \end{equation}
    and hence also $\lambda\der p\in C^\infty (S^* M)$. If $p$ is not smooth then $\lambda \der p$ is not smooth so we must have $p\in C^\infty(M)$.
\end{proof}

\appendix

\section{Coordinate formulas for derivatives on the unit cosphere bundle}
\label{sec:appendix}

We derive the coordinate formulas on a smooth Riemannian manifold $(M,g)$ where the metric $g$ is smooth up to the boundary. From the derivation we see that some of the formulas degenerate up to the boundary in gas giant geometry so the coordinate formulas are well-defined only in the interior $M^\circ$ of a gas giant. This section aims to be partially a cotangent bundle counterpart of~\cite[Appendix A]{PSU2015} so we will omit all the details which follow the same proof as in the case of the tangent bundle, cf.\ also~\cite{Pat1999} for further details.

Let $z$ be a system of local coordinates on $M$. Denote the local coordinates of $T^* M$ by $(z,\zeta)$, i.e.\ the covectors can be written as 1-forms $\zeta_\mu \der z^\mu$. By taking any arbitrary path in $T^* M$ and calculating its tangent vectors at a point we get the induced local coordinates $(z,\zeta,Z,\Zeta)$ in $T(T^* M)$, i.e.\ vectors of $T(T^* M)$ are written as $Z^\mu \partial_{z^\mu } +\Zeta_\mu \partial_{\zeta_\mu}$. We introduce the vector fields $\delta_{z^\mu}=\partial_{z^\mu }+\Gamma^\nu_{\mu\lambda}\zeta_\nu \partial_{\zeta_\lambda}$. The Sasaki metric on $T^* M$ is expressed by
\begin{align}
\label{eqn:Sasaki metric}
    \langle Z^\mu \delta_{z^\mu}+\Zeta_\mu \partial_{\zeta_\mu },\, \Tilde{Z}^\mu \delta_{z^\mu}+\Tilde{\Zeta}_\mu \partial_{\zeta_\mu}\rangle=g_{\mu\nu}Z^\mu\Tilde{Z}^\nu+g^{\mu\nu}\Zeta_\mu \Tilde{\Zeta}_\nu.
\end{align}
By~\cite{Mok1977} the horizontal and vertical subbundles of $T^* M$ are spanned by $\{\delta_{z^\mu}\}_{\mu=0}^{n-1}$ and $\{ \partial_{\zeta_\mu}\}_{\mu=0}^{n-1}$ respectively.

The hypersurface $S^* M$ in $T^* M$ is given by $S^* M=f^{-1} (1)$ where $f \colon T^* M\to \R $ is the function $f(z,\zeta)=g^{\mu\nu}(z)\zeta_\mu\zeta_\nu$.

A calculation similar to the one done for the case of tangent bundle gives us 
\begin{equation}
\der f (Z^\mu \delta_{z^\mu }+\Zeta_\mu \partial_{\zeta_\mu})=2g^{\mu\nu}\zeta_\mu\Zeta_\nu
\end{equation}
By this calculation we may define
\begin{equation}
TS^* M=\{ Z^\mu \delta_{z^\mu}+\Zeta_\mu \partial_{\zeta_\mu}\in TT^* M\mid (z,\zeta)\in S^* M,\ g^{\mu\nu }\zeta_\mu \Zeta_\nu =0\}
\end{equation}
Combining the definition of $TS^* M$ with the Sasaki metric in equation~\eqref{eqn:Sasaki metric} gives the Sasaki metric on $S^* M$. Since we can identify the quantity $g^{\mu\nu}\Zeta_\nu $ as the components of a vector field on $M$, i.e.\ $v:=g^{\mu\nu}\Zeta_\nu\partial_{z^\mu}$, then in the definition of $TS^* M$ we may write
\begin{equation}
g^{\mu\nu}\zeta_\mu \Zeta_\nu =\zeta_\mu v^\mu =\zeta(v).
\end{equation}
We define the normal bundle $N\to S^* M$ by 
\begin{equation}
\label{eqn:normal bundle}
N_{(z,\zeta)}=\{ v\in T_z M\ \colon \ \zeta(v) =0\}
\end{equation}
With this discussion we see that locally we have the identifications for horizontal and vertical distributions $\mathcal{H}\to N$ and $\mathcal{V}\to N$ by $Z^\mu \delta_{z^\mu}\mapsto Z^\mu \partial_{z^\mu}$ and $\Zeta_ \mu \partial_{\zeta_ \mu }\mapsto g^{\mu\nu} \Zeta_\mu \partial_{z^\nu}$ respectively. The latter mapping should be interpreted as the change of basis $\partial_{\zeta_\mu}\mapsto g^{\mu\nu}\partial_{z^\nu}$, since $\mathcal{V}\to N$ is given by the natural isomorphism $K_{(z,\zeta)}$ for $(z,\zeta)\in T^* M$, the local version of the connection map defined in section~\ref{sec:preliminaries}.

Define $p\colon T^* M\setminus 0\to S^* M$ by $p(z,\zeta)=(z,\zeta/\abs{\zeta}_g)$. We extend the derivatives $\delta_{z^\mu}$ and $\partial_{\zeta_\mu}$ of a function $u\in C^\infty (T^* M)$ from $T^* M$ to $S^* M$ by 
\begin{equation}
\delta_{z^\mu } u =\delta_{z^\mu } (u\circ p)|_{S^* M}\teksti{and} \partial_{\zeta_\mu } u=\partial_{\zeta_\mu }(u\circ p)|_{S^* M}.
\end{equation}
We calculate the vertical gradient by considering the restriction $u_z=u|_{S^*_z M}$. Use $p$ to extend $u_z$ from $S^*_z M$ to $T^*_z M$ so that the differential is
\begin{equation}
\der (u_z\circ p)=\partial_{\zeta_\mu }(u(z,\zeta))\der \zeta_\mu =\partial_{\zeta_\mu }u \der \zeta_\mu 
\end{equation}
and the gradient is obtained by taking the musical isomorphism
\begin{equation}
\nabla (u_z\circ p)=(g_{\mu\nu}\partial_{\zeta_\nu}u)\partial_{\zeta_\mu }.
\end{equation}
This final formula is the vertical gradient $\gradv$ on $\mathcal{V}$. Taking the identification $\mathcal{V}\to N$ via $K_{(z,\zeta)}$ for $(z,\zeta)\in T^* M$ changes the basis so that on the normal bundle $N$ the vertical gradient is given by
\begin{equation}
\label{eqn:vertical gradient local}
    \gradv u
    =
    (g_{\mu\nu}\partial_{\zeta_\mu }u)g^{\nu\lambda}\partial_{z^\lambda }
    =
    \delta^\lambda_{\ \mu}\partial_{\zeta_\mu }u\partial_{z^\lambda}
    =
    \partial_{\zeta_\mu }u\partial_{z^\mu }
\end{equation}
for $u\in C^\infty (S^* M)$.
A similar computation as in the case of the unit tangent bundle $SM$ gives that the horizontal gradient is given by
\begin{equation}
\label{eqn:horizontal gradient local}
    \gradh u
    =
    (g^{\mu\nu }\delta_{z^\nu}u -(Xu)g^{\mu\nu}\zeta_\nu )\partial_{z^\mu} 
\end{equation}
for $u\in C^\infty(S^* M)$ where $X$ is the Hamiltonian vector field of the geodesic flow. By construction we have 
\begin{equation}
\gradv u(z,\zeta) \in N_{z,\zeta}
\quad\text{and}\quad
\gradh u(z,\zeta)\in N_{z,\zeta}
\end{equation}
for all $(z,\zeta)\in S^* M$, i.e. the vertical and horizontal gradients are sections of the normal bundle $N$.

\bibliographystyle{abbrv}
\bibliography{references.bib}

\end{document}